\documentclass[11pt,a4paper,reqno]{amsart}
\usepackage[includehead,includefoot,margin=24mm]{geometry}
\usepackage{times}
\usepackage{amsmath} %
\usepackage{amssymb} %
\usepackage{amsthm} %
\usepackage{thmtools}
\usepackage{enumerate} %
\usepackage{enumitem}
\usepackage{quiver}
\usepackage{color} %
\xdefinecolor{tu-blue}{RGB}{19,83,138}
\usepackage{hyperref}
\hypersetup{
    colorlinks=true,       % false: boxed links; true: colored links
    linkcolor=tu-blue,          % color of internal links
    citecolor=tu-blue,        % color of links to bibliography
    filecolor=tu-blue,      % color of file links
    urlcolor=tu-blue           % color of external links
}

\usepackage[nameinlink]{cleveref} 

\newtheorem{introthm}{Theorem}

\newtheorem{theorem}{Theorem}[section] %
\newtheorem{corollary}[theorem]{Corollary} %
\newtheorem{lemma}[theorem]{Lemma} %
\newtheorem{proposition}[theorem]{Proposition} %
{\theoremstyle{remark} %
  \newtheorem{remark}[theorem]{Remark}} %
{\theoremstyle{definition} %
  \newtheorem{definition}[theorem]{Definition} %
  \newtheorem{example}[theorem]{Example} %
}

\crefname{theorem}{Theorem}{Theorems}
\crefname{proposition}{Proposition}{Propositions}
\crefname{lemma}{Lemma}{Lemmas}
\crefname{corollary}{Corollary}{Corollaries}
\crefname{definition}{Definition}{Definitions}
\crefname{remark}{Remark}{Remarks}
\crefname{example}{Example}{Examples}

\newcommand{\PP}[0]{\ensuremath{\mathbb{P}}}
\newcommand{\CC}[0]{\ensuremath{\mathbb{C}}}
\newcommand{\RR}[0]{\ensuremath{\mathbb{R}}}
\newcommand{\ZZ}[0]{\ensuremath{\mathbb{Z}}}
\newcommand{\NN}[0]{\ensuremath{\mathbb{N}}}
\newcommand{\FF}[0]{\ensuremath{\mathbb{F}}}

\newcommand{\tvarphi}[0]{\ensuremath{\widetilde{\varphi}}}

\newcommand\h{\operatorname{H}}

\newcommand{\Aut}[0]{\ensuremath{\operatorname{Aut}}}
\newcommand{\Gal}[0]{\ensuremath{\operatorname{Gal}}}
\newcommand{\Spec}[0]{\ensuremath{\operatorname{Spec}}}
\newcommand{\Lin}[0]{\ensuremath{\operatorname{Lin}}}
\newcommand{\diag}[0]{\ensuremath{\operatorname{diag}}}
\newcommand{\PGL}[0]{\ensuremath{\operatorname{PGL}}}
\newcommand{\GL}[0]{\ensuremath{\operatorname{GL}}}

\newcommand{\Id}[0]{\ensuremath{\operatorname{Id}}}

\newcommand{\GP}[0]{\ensuremath{\operatorname{GP}}}
\newcommand{\PGP}[0]{\ensuremath{\operatorname{PGP}}}
\newcommand{\GT}[0]{\ensuremath{\operatorname{GT}}}
\newcommand{\PGT}[0]{\ensuremath{\operatorname{PGT}}}
\newcommand{\PGU}[0]{\ensuremath{\operatorname{PGU}}}
\newcommand{\Vars}[0]{\ensuremath{\operatorname{Vars}}}
\newcommand{\Spar}[0]{\ensuremath{\operatorname{Spar}}}
\newcommand{\drank}[0]{\ensuremath{\operatorname{diff.rank}}}

\begin{document}

\title[Twisted forms of classical hypersurfaces]{Twisted forms of classical hypersurfaces}

\author{Alvaro Liendo} %
\address{Instituto de Matem\'atica y F\'isica, Universidad de Talca,
  Casilla 721, Talca, Chile} %
\email{aliendo@utalca.cl}

\author{Mat\'ias V\'asquez} %
\address{Instituto de Matem\'atica y F\'isica, Universidad de Talca,
  Casilla 721, Talca, Chile} %
\email{matias.vasquez@utalca.cl}

\date{\today}

\thanks{
  \mbox{\hspace{11pt}}{\it Key words}: Smooth hypersurfaces, automorphism groups, Galois cohomology, real forms, forms over finite fields.\\
  \mbox{\hspace{11pt}} Both authors were partially supported by ANID Fondecyt Exploraci\'on 13250049 and ANID Fondecyt Regular 1240101. The second author was also supported by ANID Beca de Magister Nacional 22260199}

\begin{abstract}
  We count the twisted forms, over the field of real numbers and over
  finite fields, of the three classical families of smooth hypersurfaces
  with large automorphism group: the Fermat, Delsarte and Klein
  hypersurfaces. Our main tool is a counting formula for the Galois
  cohomology set of a smooth hypersurface whose
  automorphism group is the semidirect product of a diagonal abelian group
  and a group of permutations of the monomials of its defining equation. In
  order to apply this formula over finite fields, we extend the
  differential method of Oguiso and Yu to positive
  characteristic, and we compute the automorphism groups of the Fermat,
  Delsarte and Klein hypersurfaces over algebraically closed fields of positive
  characteristic under explicit arithmetic conditions on $p$. Over
  the reals, our count recovers a recent theorem of Sasaki on the real
  forms of Fermat hypersurfaces.
\end{abstract}

\maketitle

\section*{Introduction}

Let $X\subset\PP^{n+1}$ be a smooth hypersurface of dimension
$n\geq2$ and degree $d\geq3$ with $(n,d)\neq(2,4)$ over an algebraically
closed field. A classical theorem of Matsumura and Monsky \cite{MM64}
asserts that the automorphism group $\Aut(X)$ is finite and that every
automorphism of $X$ is the restriction of a linear automorphism of the
ambient projective space; Benoist \cite{Ben13} refined this by showing that
the automorphism group scheme is reduced in every characteristic. A
general hypersurface has trivial automorphism group, but special
hypersurfaces can have large automorphism groups, see for instance
\cite{GL11,GL13,OY19,GALM22,Zhe22,MR4771229} and the references therein.
With finitely many exceptions, the smooth hypersurface of dimension $n$ and
degree $d$ with the largest automorphism group is the Fermat hypersurface
\cite{esserlargeautomorphismgroups,yang2025automorphismgroupssmoothhypersurfaces},
while the Klein hypersurface is, up to isomorphism, the unique one
admitting an automorphism of the largest possible prime order \cite{GL13}.

The results above concern hypersurfaces over an algebraically closed
field. Over a non-closed field the natural counterpart is the study of
forms, this is, given a Galois extension $K/k$ and a variety $X$ over $K$, a
$k$-form of $X$ is a variety $X'$ over $k$ such that
$X'\times_{\Spec(k)}\Spec(K)$ is $K$-isomorphic to $X$. By Weil descent,
the set of $k$-forms of $X$ up to $k$-isomorphism is classified by the
Galois cohomology set $H^1(\Gal(K/k),\Aut_K(X))$ \cite{BoSe64,Se}. Real
forms, i.e., the case $k=\RR$ and $K=\CC$, have been studied
recently: there exist smooth projective varieties, and
even projective rational surfaces, with infinitely many non-isomorphic real
forms \cite{Lesieutre,DiOg}, and rational surfaces with uncountably many
\cite{Anna}; in the other direction, several finiteness results are known
\cite{Benzerga,Cattaneo}. For a smooth hypersurface as above the
automorphism group is finite, so the number of its forms is finite, and it
becomes a natural problem to compute this number exactly. For Fermat
hypersurfaces over the reals this was recently achieved by Sasaki
\cite{Sasaki}.

In this paper we count the $k$-forms of the three classical families
of hypersurfaces with large symmetry, namely the Fermat, Delsarte and Klein
hypersurfaces $F^n_d=V(\mathcal F)$, $T^n_d=V(\mathcal T)$ and
$K^n_d=V(\mathcal K)$ given by
$$\mathcal F=\sum_{i=0}^{n+1}x_i^{d},\qquad
\mathcal T=\sum_{i=0}^{n}x_i^{d-1}x_{i+1}+x_{n+1}^d,\qquad
\mathcal K=\sum_{i\in\ZZ/(n+2)\ZZ}x_i^{d-1}x_{i+1},$$
over the fields $k=\RR$ and $k=\FF_q$ with $q$ odd. What these families
have in common is that their automorphism groups decompose as a semidirect
product $D\rtimes P$, where $D$ is an abelian group of diagonal
automorphisms and $P\leq S_{n+2}$ is the group of permutations of the
variables preserving the set of monomials of the defining equation
(\cref{cor:DrtimesP}). Our first main result is a counting formula for the
forms of any such hypersurface; here $\Gamma=\Gal(K/k)$, the twisted sets
$H^1(\Gamma,\ _cD)$ are computable since $D$ is abelian, and the key point
is that the section $P\to\Aut(X)$ given by permutation matrices is
$\Gamma$-equivariant.

\begin{introthm}[See \cref{thm:A}]\label{introthm:A}
Let $X=V(F)$ be a smooth hypersurface defined over $k$, all of whose
monomials have coefficient $1$, and such that $\Aut(X)\cong D\rtimes P$ as
above. Then
$$|H^1(\Gamma,\Aut(X))|=\sum_{[c]\in H^1(\Gamma,P)}
\big|H^1(\Gamma,\ _{c}D)/C_P(c)\big|,$$
where $C_P(c)$ denotes the centralizer of $c$ in $P$.
\end{introthm}

Over the reals, where $H^1(\Gamma,P)$ is the set of elements of order
at most two of $P$ up to conjugation, the formula yields completely
explicit answers (see
\cref{cor:R_forms_Delsarte,cor:R_forms_Klein,cor:R_forms_Fermat}): for
$d\geq4$, the Delsarte hypersurface $T^n_d$ has exactly two real forms if
$d$ is odd and only one if $d$ is even, and the Klein hypersurface $K^n_d$
has one, two or four real forms according to whether $n$ is odd, $n$ is
even and $d$ is odd, or $n$ and $d$ are both even; for the Fermat
hypersurface ($d\geq3$) we recover the theorem of Sasaki \cite{Sasaki}:
$$ |H^1(\Gamma, \Aut(F^n_d))| = \begin{cases}
\left\lfloor n/2 \right\rfloor+ 2 & \text{if $d$ is odd,}\\
(n+3)(n+5)/8 & \text{if $d$ is even and $n$ odd,}\\
1 + (n+4)(n+6)/8 & \text{if $d$ and $n$ are even.}
\end{cases}$$

Over a finite field $k=\FF_q$ the group $\Gamma$ is topologically
generated by the Frobenius and $H^1(\Gamma,P)$ becomes the set of conjugacy
classes of $P$, so the same formula applies; but before using it one must
know $\Aut(X)$ over $K=\overline{\FF}_q$. We therefore extend the
differential method of Oguiso and Yu \cite{OY19}, in the refined form of
\cite{MR4771229}, to arbitrary characteristic (\cref{thm:GT}), making
explicit the hypotheses on $p$ that are needed. As a consequence we
show that the automorphism groups of the classical hypersurfaces over
$\overline{\FF}_q$ coincide with those over $\CC$, namely
$$\Aut(F^n_d)\cong(\ZZ/d\ZZ)^{n+1}\rtimes S_{n+2},\quad
\Aut(T^n_d)\cong\ZZ/(d-1)^{n+1}\ZZ,\quad
\Aut(K^n_d)\cong(\ZZ/m\ZZ)\rtimes\ZZ/(n+2)\ZZ,$$
with $m=\frac{(d-1)^{n+2}-(-1)^{n+2}}{d}$, under the respective assumptions
$p\nmid d(d-1)$, $p\nmid d(d-1)(d-2)$ and $p\nmid d(d-1)(d-2)m$ (and
$n\geq4$ for the Klein cubic), see
\cref{prop:fermat,prop:delsarte,prop:klein,prop:kleincubic}. The condition
$p\nmid d-1$ cannot simply be removed: when $d=p^e+1$ the Fermat equation
is a Hermitian form and $\Aut(F^n_d)$ contains the much larger unitary
group $\PGU_{n+2}(\FF_{p^{2e}})$, see \cref{rem:hermitian}. In passing, we
repair a small gap in the proof of \cite[Theorem~2.6]{MR4771229}, whose
statement remains unchanged, see \cref{rem:gap}.

With these ingredients, the counting formula gives that the Delsarte
hypersurface has exactly $\gcd(q-1,(d-1)^{n+1})$ forms over $\FF_q$
(\cref{cor:Fq_forms_Delsarte}), while for the Klein and Fermat
hypersurfaces we obtain closed formulas for the cardinalities of all the
twisted sets $H^1(\Gamma,\ _cD)$, reducing the count of $\FF_q$-forms to an
explicit orbit count, see
\cref{cor:Fq_forms_Klein,cor:Fq_forms_Fermat}.

The paper is organized as follows. In \cref{sec:diff} we develop the
differential method in positive characteristic and compute the
automorphism groups of the Fermat, Delsarte and Klein hypersurfaces over
any algebraically closed field. In \cref{sec:cohomology} we collect the
tools from non-abelian Galois cohomology that we need, following \cite{Se}.
In \cref{sec:forms} we prove Theorem A and carry out the counts of real
and $\FF_q$-forms of the classical hypersurfaces.

\subsection*{Acknowledgements}

The authors thank Giancarlo Lucchini Arteche for helpful discussions during the preparation of this paper.

\section{The differential method in positive characteristic}\label{sec:diff}

Throughout this section $K$ denotes an algebraically closed field of
characteristic $p\geq 0$ and $V$ a $K$-vector space of dimension $n+2$ with
$n\geq 2$. We denote by $\pi\colon\GL(V)\to\PGL(V)$ the canonical projection;
for $\varphi\in\PGL(V)$ we write $\tvarphi\in\GL(V)$ for a chosen lift and
$\tvarphi^*\colon S(V^*)\to S(V^*)$, $\tvarphi^*(F)=F\circ\tvarphi$, for the
induced graded automorphism of the symmetric algebra. We fix dual bases
$\beta=\{e_0,\ldots,e_{n+1}\}$ of $V$ and $\beta^*=\{x_0,\ldots,x_{n+1}\}$ of
$V^*$ and we identify $S(V^*)\simeq K[x_0,\ldots,x_{n+1}]$. For a hypersurface
$X=V(F)\subset\PP(V)$ given by $F\in S^d(V^*)$ we let
$$\Lin(X)=\{\varphi\in\PGL(V)\mid \varphi(X)=X\}
=\{\varphi\in\PGL(V)\mid \tvarphi^*(F)=\lambda F\ \text{for some}\ \lambda\in K^*\}.$$

The differential method of Oguiso and Yu \cite{OY19}, in the refined form
developed in \cite{MR4771229}, is of a purely formal nature and a
large part of it carries over to positive characteristic unchanged. The aim
of this section is to review it over $K$, indicating where the
characteristic matters, and to record the description of the
automorphism groups of the Fermat, Delsarte and Klein hypersurfaces over $K$
that will be used in the following sections. We refer to \cite{MR4771229} for
every proof that is characteristic-free.

Over $\CC$, the starting point of the differential method is the classical
theorem of Matsumura and Monsky \cite{MM64} asserting that all the
automorphisms of a smooth hypersurface are linear and that there are finitely
many of them. This theorem is in fact valid over
any algebraically closed field; the refinement that we need, due
to Benoist, is that the automorphism group \emph{scheme} is moreover reduced,
and that the list of exceptional cases is exactly the classical one.

\begin{theorem}[{\cite[Thm.~3.1]{Ben13}}]\label{thm:MM}
  Let $k$ be a field of characteristic $p\geq0$ and let $X\subset\PP^{n+1}_{k}$
  be a smooth hypersurface of dimension $n\geq2$ and degree $d\geq3$ with
  $(n,d)\neq(2,4)$. Then the group schemes $\Aut_k(X)$ and
  $\Aut_k(X,\mathcal O(1))$ coincide and are finite and reduced. In particular,
  over an algebraically closed field, $\Aut(X)=\Lin(X)$ is a finite group.
\end{theorem}

\begin{remark}\label{rem:reduced}
  Two comments on \cref{thm:MM}, both invisible in characteristic zero.
  First, the excluded cases are exactly the classical ones: in
  \cite[Thm.~3.1]{Ben13} the hypersurfaces of dimension $n\geq2$ appear in the
  exceptional list only as quadrics ($d=2$) and quartic surfaces
  ($(n,d)=(2,4)$, where $\Aut_k(X)$ may be infinite although
  $\Aut_k(X,\mathcal O(1))$ is finite and reduced); no new exception arises in
  positive characteristic. Second, the
  reducedness of $\Aut_k(X)$ means that $X$ has no infinitesimal automorphisms,
  i.e.\ $\h^0(X,T_X)=0$; equivalently $\Aut_{K}(X)$ is a finite \'etale group
  scheme for $K=\overline k$ (recall that $k$ is perfect in our applications).
  In particular $\Aut_K(X)$ is a finite (abstract) group equipped with a
  continuous $\Gamma$-action, which is the setting used in the following
  sections, where $H^1(\Gal(K/k),\Aut_K(X))$ classifies the $k$-forms of $X$.
\end{remark}

From now on we assume $n\geq2$; this excludes plane curves, for which \cref{thm:MM} would require the extra input of \cite[Thm.~1]{Cha78}, and
it costs nothing since all the hypersurfaces studied below have dimension at
least $2$.

We first recall the directional derivative operator attached to a form and its
rank.

\begin{definition}[{\cite[Definition 1.2]{MR4771229}}]\label{def:drank}
  Let $\mathcal{D}$ be the directional derivative operator
  $$\mathcal{D}\colon V^*\times S(V^*)\to S(V^*),\qquad (x,F)\mapsto
  \frac{\partial F}{\partial x}=\nabla(F)\cdot x,$$
  and, for a fixed homogeneous $F\in S(V^*)$, let
  $\mathcal{D}_F\colon V^*\to S(V^*)$, $x\mapsto \partial F/\partial x$ be its
  specialization. The \emph{differential rank} of $F$ is
  $\drank(F):=\operatorname{rank}\mathcal{D}_F$.
\end{definition}

\begin{proposition}[{\cite[Proposition 1.3]{MR4771229}}]\label{prop:drank}
  Let $F,G\in S(V^*)$ and let $\tvarphi\in\GL(V)$ be such that
  $\tvarphi^*(G)=\lambda F$ for some $\lambda\in K^*$. Then
  $\drank(F)=\drank(G)$. More precisely,
  \begin{equation}\label{eq:chain}
    \lambda\,\frac{\partial F}{\partial x}
    =\tvarphi^*\left(\frac{\partial G}{\partial(\tvarphi^*x)}\right)
    \qquad\text{for every }x\in V^*.
  \end{equation}
\end{proposition}

\begin{proof}
  The proof in \cite{MR4771229} only uses the chain rule for the composition
  with a linear map, which is a formal identity of derivations valid over any
  commutative ring, and the fact that $\tvarphi^*$ is an isomorphism. Hence it
  is valid over $K$.
\end{proof}

\begin{corollary}[{\cite[Corollary 1.4]{MR4771229}}]\label{cor:drank}
  Let $X=V(F)\subset\PP(V)$ and let $\varphi\in\PGL(V)$ be a linear
  automorphism of $X$. Then for every $x\in V^*$ the forms
  $\partial F/\partial x$ and $\partial F/\partial(\tvarphi^*x)$ differ by
  $\tvarphi^*$ and a nonzero scalar; in particular
  $$\drank\left(\frac{\partial F}{\partial x}\right)=
  \drank\left(\frac{\partial F}{\partial (\tvarphi^*x)}\right).$$
\end{corollary}

To exploit \cref{cor:drank} we need to control the combinatorics of the
monomials of $F$, and we keep for this the notions of
\cite[Definitions 2.1 and 2.2]{MR4771229}: a matrix is a
\emph{generalized permutation matrix} if it has at most one nonzero entry in
each row and each column, and a \emph{generalized triangular matrix} if it is of
the form $P_1TP_2$ with $P_1,P_2$ permutation matrices and $T$ upper triangular;
we write $\GP(V,\beta)$, $\GT(V,\beta)$ for the corresponding subsets of
$\GL(V)$ and $\PGP(V,\beta)$, $\PGT(V,\beta)$ for their images in $\PGL(V)$.
The \emph{distance} between two monomials $\prod x_i^{a_i}$ and $\prod
x_i^{b_i}$ is $\sum_i|a_i-b_i|$, the \emph{sparsity} $\Spar(F)$ of $F$ is the
minimum of the distances between two distinct monomials of $F$ (with the
convention $\Spar(F)=\infty$ if $F$ has at most one monomial), and
$\Vars(F)\subseteq\beta^*$ is the set of variables occurring in $F$.

In positive characteristic a variable may occur in $F$ and yet not be seen by
the operator $\mathcal{D}_F$, as it happens with $x_0$ in $F=x_0^p$. This is the only
difference with the characteristic zero case.

\begin{definition}\label{def:effvars}
  Let $F\in S(V^*)$ be homogeneous. The set of \emph{effective variables} of
  $F$ is
  $$\Vars_p(F):=\left\{x\in\beta^*\ \Big|\ \frac{\partial F}{\partial x}\neq0
  \right\}\subseteq\Vars(F).$$
\end{definition}

If $p=0$ then $\Vars_p(F)=\Vars(F)$. With this notation, Lemma 2.3 of
\cite{MR4771229} holds over $K$ without any hypothesis on the characteristic.

\begin{lemma}\label{lem:rank}
  Let $F\in S(V^*)$ be homogeneous of degree $d\geq2$ with $\Spar(F)>2$. Then
  $\drank(F)=\#\Vars_p(F)$.
\end{lemma}

\begin{proof}
  Let $i\neq j$. If $\partial F/\partial x_i$ and $\partial F/\partial x_j$ had
  a common monomial, there would exist monomials $x^a\neq x^b$ of $F$ with
  $x^{a-e_i}=x^{b-e_j}$, and then the distance between $x^a$ and $x^b$ would be
  $2$, contradicting $\Spar(F)>2$. Hence the nonzero elements of
  $\{\partial F/\partial x_i\}_{i}$ have pairwise disjoint supports and so they
  are linearly independent. As $\mathcal{D}_F(\sum_ic_ix_i)=\sum_ic_i\,\partial
  F/\partial x_i$, we get that $\ker \mathcal{D}_F$ is the span of
  $\beta^*\setminus\Vars_p(F)$, whence $\operatorname{rank}\mathcal{D}_F=\#\Vars_p(F)$.
\end{proof}

\begin{definition}\label{def:poset}
  We endow $\beta^*$ with the relation $\leq_F$ given by
  $$x_i\leq_F x_j\iff \Vars_p\left(\frac{\partial F}{\partial x_i}\right)
  \subseteq\Vars_p\left(\frac{\partial F}{\partial x_j}\right).$$
  As in \cite[Remark 2.5]{MR4771229}, $(\beta^*,\leq_F)$ need not be a partial
  order; we say that a poset $(\beta^*,\leq_F)$ is \emph{trivial} if
  $x_i\leq_Fx_j$ implies $i=j$.
\end{definition}

\begin{theorem}\label{thm:GT}
  Let $X=V(F)\subset\PP(V)$ be a smooth hypersurface of dimension $n\geq2$ and
  degree $d\geq3$ with $(n,d)\neq(2,4)$, and assume $\Spar(F)>4$.
  Assume moreover the following condition $(\star)$, which is automatic
  when $p=0$ and holds for every $p$-tame form (see \cref{lem:tame} below):
  for every $i$ such that $\partial^2F/\partial x_i^2=0$, every monomial of
  $F$ has degree at most one in $x_i$.
  \begin{enumerate}
    \item If $(\beta^*,\leq_F)$ is a poset, then $\Aut(X)\subseteq\PGT(V,\beta)$.
      More precisely, for every $\varphi\in\Aut(X)$ there is a permutation
      $\tau$ of $\{0,\ldots,n+1\}$ with
      $\drank(\partial F/\partial x_i)=\drank(\partial F/\partial x_{\tau(i)})$
      such that
      \begin{equation}\label{eq:shape}
        \tvarphi^*x_i=\tvarphi_{i\tau(i)}x_{\tau(i)}
        +\sum_{x_\ell<_Fx_{\tau(i)}}\tvarphi_{i\ell}x_\ell ,
        \qquad \tvarphi_{i\tau(i)}\neq0 .
      \end{equation}
    \item If $(\beta^*,\leq_F)$ is trivial, then $\Aut(X)\subseteq\PGP(V,\beta)$.
  \end{enumerate}
\end{theorem}

\begin{proof}
  By \cref{thm:MM} we have $\Aut(X)=\Lin(X)\subseteq\PGL(V)$, and the
  proof of \cite[Theorem 2.6, Remark 2.7 and Corollary 2.8]{MR4771229} goes
  through once \cite[Lemma 2.3]{MR4771229} is replaced by \cref{lem:rank}. Indeed, the only input is the identity
  \begin{equation}\label{eq:vars-union}
    \drank\left(\frac{\partial F}{\partial(\tvarphi^*x_i)}\right)
    =\#\!\!\bigcup_{\{j\,\mid\,\tvarphi_{ij}\neq0\}}\!\!
    \Vars_p\left(\frac{\partial F}{\partial x_j}\right),
  \end{equation}
  which we now check; this is where our argument differs from
  the proof of \cite[Theorem 2.6]{MR4771229}, and it is the only place where
  smoothness and condition $(\star)$ are used, see \cref{rem:gap} below.

  Write $G:=\partial F/\partial(\tvarphi^*x_i)=\sum_j c_j\,\partial
  F/\partial x_j$ with $c_j:=\tvarphi_{ij}$, and set
  $M_{jk}:=\partial^2F/\partial x_j\partial x_k=M_{kj}$. If
  $\{j,k\}\neq\{j',k'\}$ as unordered pairs, then $M_{jk}$ and $M_{j'k'}$
  have disjoint supports: a common monomial yields monomials $x^a$, $x^b$ of
  $F$ with $a-e_j-e_k=b-e_{j'}-e_{k'}$, so either $a=b$ and
  $\{j,k\}=\{j',k'\}$, or $a\neq b$ and their distance is at most $4$,
  contradicting $\Spar(F)>4$. In particular, for fixed $k$ no cancellation
  occurs in $\partial G/\partial x_k=\sum_j c_jM_{jk}$, whence
  $x_k\in\Vars_p(G)$ if and only if $M_{jk}\neq0$ for some $j$ with
  $c_j\neq0$; that is,
  $$\Vars_p(G)=\bigcup_{\{j\,\mid\,c_j\neq0\}}
  \Vars_p\left(\frac{\partial F}{\partial x_j}\right).$$
  It remains to show that $\drank(G)=\#\Vars_p(G)$, i.e.\ that the forms
  $\partial G/\partial x_k$ with $x_k\in\Vars_p(G)$ are linearly independent.
  Suppose that $\sum_k t_k\,\partial G/\partial x_k=0$ and that
  $t_{k_0}\neq0$ for some $x_{k_0}\in\Vars_p(G)$. Grouping
  $\sum_{j,k}t_kc_jM_{jk}$ by unordered pairs and using the disjointness of
  the supports, we obtain
  $$(t_kc_j+t_jc_k)\,M_{jk}=0\ \ \text{for all }j\neq k,
  \qquad t_kc_k\,M_{kk}=0\ \ \text{for all }k.$$
  Pick $j_0$ with $c_{j_0}\neq0$ and $M_{j_0k_0}\neq0$. If $j_0=k_0$, the
  second relation gives $t_{k_0}c_{k_0}=0$, a contradiction; hence
  $j_0\neq k_0$ and $t_{k_0}c_{j_0}+t_{j_0}c_{k_0}=0$, which forces
  $c_{k_0}\neq0$ and $t_{j_0}\neq0$, and then $M_{k_0k_0}=M_{j_0j_0}=0$. By
  condition $(\star)$, every monomial of $F$ has degree at most one in
  $x_{j_0}$ and in $x_{k_0}$, while $M_{j_0k_0}\neq0$ provides a monomial of
  $F$ divisible by $x_{j_0}x_{k_0}$. Consider the line
  $L=\PP(\langle e_{j_0},e_{k_0}\rangle)\subseteq\PP(V)$. Since $d\geq3$, no
  monomial of $F$ is supported in $\{x_{j_0},x_{k_0}\}$, so $L\subseteq X$.
  A monomial of $\partial F/\partial x_\ell$ supported in
  $\{x_{j_0},x_{k_0}\}$ has degree $d-1\geq2$ and degree at most one in each
  of $x_{j_0}$, $x_{k_0}$; hence for $d\geq4$ there is none and all the
  partial derivatives of $F$ vanish along $L$, while for $d=3$ the
  restriction of every partial derivative to $L$ is a multiple of
  $x_{j_0}x_{k_0}$ and they all vanish at the point $[e_{j_0}]\in L$. In both
  cases $X$ acquires a singular point, contradicting smoothness. Hence
  $\ker \mathcal{D}_G$ is the span of $\beta^*\setminus\Vars_p(G)$ and
  $\drank(G)=\#\Vars_p(G)$, which is \eqref{eq:vars-union}.
\end{proof}

\begin{remark}\label{rem:gap}
  The proof of \cite[Theorem 2.6]{MR4771229} deduces the identity
  \eqref{eq:vars-union} from the inequality
  $\Spar\big(\partial F/\partial(\tvarphi^*x_i)\big)>2$, which does not
  follow from $\Spar(F)>4$: two partial derivatives of one and the same
  monomial of $F$ lie at distance $2$. For instance, if
  $F=x_0^{d-1}x_1+x_1^{d-1}x_2+\cdots+x_{n+1}^{d-1}x_0$ is the equation of
  the Klein hypersurface $K^n_d$, then $\partial F/\partial(x_0+x_1)$
  contains both $x_0^{d-2}x_1$ and $x_0^{d-1}$. The proof of \cref{thm:GT}
  given above repairs this gap. Since condition $(\star)$ is automatic in
  characteristic zero, the statement of \cite[Theorem 2.6]{MR4771229} is
  unaffected. In positive characteristic, however, condition $(\star)$
  cannot be dispensed with: if $p\mid d-2$, then
  $\partial^2F/\partial x_i^2=(d-1)(d-2)x_i^{d-3}x_{i+1}$ vanishes and
  \eqref{eq:vars-union} genuinely fails, e.g.\
  $\drank\big(\partial F/\partial(x_0+x_1)\big)=3$ while the right-hand side
  of \eqref{eq:vars-union} equals $4$.
\end{remark}

In order to compute $(\beta^*,\leq_F)$ as in characteristic zero it suffices to
know that $\Vars_p$ and $\Vars$ agree for $F$ and for its first partial
derivatives. This is guaranteed by the following condition, which is the only
restriction on $d$ that we shall impose.

\begin{definition}\label{def:tame}
  A homogeneous form $F\in S(V^*)$ is \emph{$p$-tame} (with respect to $\beta$)
  if for every monomial $x^a$ of $F$ and every index $i$ we have
  $$a_i\geq1\ \Longrightarrow\ p\nmid a_i,\qquad\text{and}\qquad
  a_i\geq2\ \Longrightarrow\ p\nmid a_i-1 .$$
  Every form is $p$-tame when $p=0$.
\end{definition}

\begin{lemma}\label{lem:tame}
  If $F$ is $p$-tame, then $\Vars_p(F)=\Vars(F)$ and
  $\Vars_p(\partial F/\partial x)=\Vars(\partial F/\partial x)$ for every
  $x\in\beta^*$; moreover the monomials of $\partial F/\partial x_j$ are exactly
  the $x^{a-e_j}$ with $x^a$ a monomial of $F$ and $a_j\geq1$. In particular
  $(\beta^*,\leq_F)$ is the poset computed in characteristic zero, and
  $F$ satisfies condition $(\star)$ of \cref{thm:GT}.
\end{lemma}

\begin{proof}
  For a fixed $j$ the monomials $x^{a-e_j}$, with $x^a$ running over the
  monomials of $F$ with $a_j\geq1$, are pairwise distinct, so no cancellation
  occurs in $\partial F/\partial x_j=\sum_a c_a\,a_j\,x^{a-e_j}$ and every
  coefficient $c_aa_j$ is nonzero. The same argument applied to
  $\partial^2F/\partial x_j\partial x_i$, whose coefficients are
  $c_a\,a_j(a_j-1)$ if $i=j$ and $c_a\,a_ja_i$ if $i\neq j$, gives the second
  assertion. In particular, if a monomial $x^a$ of $F$ has $a_i\geq2$,
  then $\partial^2F/\partial x_i^2$ contains the monomial $x^{a-2e_i}$ with
  coefficient $c_a\,a_i(a_i-1)\neq0$; hence $\partial^2F/\partial x_i^2=0$
  forces every monomial of $F$ to have degree at most one in $x_i$, which is
  condition $(\star)$ of \cref{thm:GT}.
\end{proof}

\begin{remark}\label{rem:tame-degree}
  All the forms considered below have their exponents in $\{0,1,d-1,d\}$; such a
  form is $p$-tame as soon as $p\nmid d(d-1)(d-2)$. More precisely, the
  condition $p\nmid d-2$ is needed only when the exponent $d-1$ occurs (Klein
  and Delsarte forms) and the condition $p\nmid d$ only when the exponent $d$
  occurs; for the Fermat form, whose exponents lie in $\{0,d\}$, being $p$-tame
  is equivalent to $p\nmid d(d-1)$.
\end{remark}

We can now apply the above to the classical hypersurfaces and determine their
automorphism groups over $K$. Fix $d\geq3$ and consider in $\PP^{n+1}_{K}$ the Fermat, Delsarte and Klein
hypersurfaces $F^n_d=V(\mathcal F)$, $T^n_d=V(\mathcal T)$, $K^n_d=V(\mathcal K)$
given by
$$\mathcal F=\sum_{i=0}^{n+1}x_i^{d},\qquad
\mathcal T=\sum_{i=0}^{n}x_i^{d-1}x_{i+1}+x_{n+1}^d,\qquad
\mathcal K=\sum_{i\in\ZZ/(n+2)\ZZ}x_i^{d-1}x_{i+1},$$
and set, as usual, $m=\frac{(d-1)^{n+2}-(-1)^{n+2}}{d}$. A direct computation
gives
\begin{equation}\label{eq:spar}
  \Spar(\mathcal F)=2d,\qquad \Spar(\mathcal T)=\Spar(\mathcal K)=2d-2 .
\end{equation}
We first record the smoothness of these hypersurfaces over $K$.

\begin{lemma}\label{lem:smooth}
  Let $d\geq3$. Then:
  \begin{enumerate}
    \item $F^n_d$ is smooth if and only if $p\nmid d$;
    \item if $p\nmid d(d-1)$, then $T^n_d$ is smooth;
    \item if $p\nmid dm$, then $K^n_d$ is smooth.
  \end{enumerate}
\end{lemma}

\begin{proof}
  (1) is clear since the partial derivatives of $\mathcal F$ are $dx_i^{d-1}$.

  (2) Let $P$ be a common zero of all the partial derivatives
  $\partial\mathcal T/\partial x_0=(d-1)x_0^{d-2}x_1$,
  $\partial\mathcal T/\partial x_i=(d-1)x_i^{d-2}x_{i+1}+x_{i-1}^{d-1}$ for
  $1\leq i\leq n$, and
  $\partial\mathcal T/\partial x_{n+1}=x_n^{d-1}+dx_{n+1}^{d-1}$. If
  $x_{n+1}=0$ at $P$, the last equation gives $x_n=0$ and then the previous
  ones give successively $x_{n-1}=\cdots=x_0=0$, which is impossible. Hence we
  may normalize $x_{n+1}=1$, so that $x_n^{d-1}=-d\neq0$ and $x_n\neq0$;
  descending with the same equations we get $x_i\neq0$ for all $i$, which
  contradicts $\partial\mathcal T/\partial x_0=(d-1)x_0^{d-2}x_1=0$.

  (3) Let $P$ be a common zero of
  $\partial\mathcal K/\partial x_i=(d-1)x_i^{d-2}x_{i+1}+x_{i-1}^{d-1}$. If
  $x_i=0$ for some $i$, then $x_{i-1}^{d-1}=0$ and, going around the cycle,
  $P=0$. Hence all the coordinates of $P$ are nonzero, and multiplying the
  $n+2$ equations $(d-1)x_i^{d-2}x_{i+1}=-x_{i-1}^{d-1}$ we obtain
  $(d-1)^{n+2}\prod_i x_i^{d-1}=(-1)^{n+2}\prod_i x_i^{d-1}$, that is,
  $dm=(d-1)^{n+2}-(-1)^{n+2}=0$ in $K$, contradicting $p\nmid dm$.
\end{proof}

\begin{remark}\label{rem:smooth-conv}
  The converses of (2) and (3) do not hold: for instance $T^2_3$ and $K^2_3$
  are smooth in characteristic $3$, although $3\mid d(d-1)$ and
  $3\mid dm=15$ respectively. This is harmless, since
  $p\nmid d(d-1)$, resp.\ $p\nmid dm$, will be assumed anyway.
\end{remark}

\begin{proposition}[Fermat hypersurfaces]\label{prop:fermat}
  Let $n\geq2$ and $d\geq3$ with $(n,d)\neq(2,4)$. If $p\nmid d(d-1)$,
  then
  $$\Aut(F^n_d)\cong(\ZZ/d\ZZ)^{n+1}\rtimes S_{n+2}.$$
\end{proposition}

\begin{proof}
  By \cref{lem:smooth} the hypersurface is smooth and by
  \cref{rem:tame-degree} the form $\mathcal F$ is $p$-tame, so
  \cref{lem:tame} gives
  $\Vars_p(\partial\mathcal F/\partial x_i)=\Vars(dx_i^{d-1})=\{x_i\}$ and the
  poset $(\beta^*,\leq_{\mathcal F})$ is trivial. Since $\Spar(\mathcal F)=2d>4$,
  \cref{thm:GT} yields $\Aut(F^n_d)\subseteq\PGP(V,\beta)$, and we
  conclude as in \cite[Proposition 3.1]{MR4771229}: writing $\tvarphi$ as a
  diagonal matrix times a permutation matrix and normalizing $\lambda=1$ (which
  is possible since $K$ is algebraically closed), the diagonal part satisfies
  $\tvarphi_{ii}^d=1$, i.e.\ $\tvarphi_{ii}\in\mu_d(K)$. As $p\nmid d$ the group
  $\mu_d(K)$ is cyclic of order $d$, so we get an epimorphism
  $(\ZZ/d\ZZ)^{n+2}\to D$ onto the group $D$ of diagonal automorphisms, with
  kernel the diagonal $\Delta\cong\ZZ/d\ZZ$; hence $D\cong(\ZZ/d\ZZ)^{n+1}$.
  Finally every permutation matrix preserves $\mathcal F$, so $P=S_{n+2}$.
\end{proof}

\begin{remark}\label{rem:hermitian}
  The hypothesis $p\nmid d-1$ cannot be removed. If $d=p^{e}+1$, then
  $\partial\mathcal F/\partial x_i=x_i^{p^e}$ and hence
  $\Vars_p(\partial\mathcal F/\partial x_i)=\emptyset$ for every $i$; the
  relation $\leq_{\mathcal F}$ is then not antisymmetric and
  \cref{thm:GT} does not apply. This is not a defect of the method:
  in that case $\mathcal F$ is the Hermitian form of $\FF_{p^{2e}}$, so
  $\Aut(F^n_d)$ contains $\PGU_{n+2}(\FF_{p^{2e}})$ and is in particular not contained in $\PGP(V,\beta)$.
\end{remark}

\begin{remark}\label{rmk:tilde_D}
For the Fermat hypersurface we denote by $\widetilde{D} \leq \GL(V)$ the group of diagonal automorphisms, and by $D \leq \PGL(V)$ its image under the projection $\pi\colon \GL(V) \to \PGL(V)$.
\end{remark}

\begin{proposition}[Klein hypersurfaces]\label{prop:klein}
  Let $n\geq2$ and $d\geq4$ with $(n,d)\neq(2,4)$. If $p\nmid d(d-1)(d-2)m$,
  then
  $$\Aut(K^n_d)\cong(\ZZ/m\ZZ)\rtimes\ZZ/(n+2)\ZZ.$$
\end{proposition}

\begin{proof}
  By \cref{lem:smooth} the hypersurface is smooth and $\mathcal K$ is
  $p$-tame, so by \cref{lem:tame}
  $$\Vars_p\left(\frac{\partial\mathcal K}{\partial x_i}\right)
  =\Vars\big((d-1)x_i^{d-2}x_{i+1}+x_{i-1}^{d-1}\big)
  =\{x_{i-1},x_i,x_{i+1}\}.$$
  As $n+2\geq4$, these $n+2$ subsets of cardinality $3$ are pairwise distinct
  and pairwise incomparable, so $(\beta^*,\leq_{\mathcal K})$ is trivial; since
  $\Spar(\mathcal K)=2d-2>4$, \cref{thm:GT} gives
  $\Aut(K^n_d)\subseteq\PGP(V,\beta)$.

  Let now $\varphi\in\Aut(K^n_d)$ and write $\tvarphi=\tilde\mu\tilde P_\sigma$
  with $\tilde\mu$ diagonal and $\sigma\in S_{n+2}$. Since $\tilde\mu^*$ scales
  each monomial, $\sigma$ permutes the monomials of $\mathcal K$; as $d-1\neq1$,
  the monomial $x_i^{d-1}x_{i+1}$ is sent to $x_{\sigma(i)}^{d-1}x_{\sigma(i+1)}$,
  which forces $\sigma(i+1)=\sigma(i)+1$ for all $i$, i.e.\ $\sigma$ is a power
  of the cyclic permutation $\nu=(0\,1\,\cdots\,n+1)$. Thus $P=\langle\nu\rangle
  \cong\ZZ/(n+2)\ZZ$. For the diagonal part, as in
  \cite[Proposition 3.3]{MR4771229}, normalizing $\lambda=1$ we get
  $\tvarphi_{ii}^{d-1}\tvarphi_{i+1,i+1}=1$, hence
  $\tvarphi_{ii}=\tvarphi_{00}^{(1-d)^i}$ and
  $\tvarphi_{00}^{(1-d)^{n+2}-1}=1$. Since $|(1-d)^{n+2}-1|=dm$ and $p\nmid dm$,
  the group $\mu_{dm}(K)$ is cyclic of order $dm$, so we obtain an epimorphism
  $\ZZ/dm\ZZ\to D$ whose kernel is $\mu_d(K)$, of order $d$; therefore $D$ is
  cyclic of order $m$.
\end{proof}

\begin{remark}\label{rem:klein-n1}
  The standing hypothesis $n\geq2$ is needed here already for the poset: for
  $n=1$ one has
  $\Vars_p(\partial\mathcal K/\partial x_i)=\{x_0,x_1,x_2\}$ for every $i$, so
  $\leq_{\mathcal K}$ is not antisymmetric.
\end{remark}

\begin{proposition}[Klein cubic hypersurfaces]\label{prop:kleincubic}
  Let $d=3$, $n\geq4$ and $m=\frac{2^{n+2}-(-1)^{n+2}}{3}$. If $p\geq5$ and
  $p\nmid m$, then
  $$\Aut(K^n_3)\cong(\ZZ/m\ZZ)\rtimes\ZZ/(n+2)\ZZ.$$
\end{proposition}

\begin{proof}
  For $d=3$ we have $\Spar(\mathcal K)=4$ and \cref{thm:GT} does not
  apply. The finer analysis of \cite[Theorem 3.5]{MR4771229} shows directly
  that $\drank(\partial\mathcal K/\partial x)=3$ if and only if $x=c_ix_i$ for
  some $i$, hence that $\Aut(K^n_3)\subseteq\PGP(V,\beta)$. That analysis
  consists of computing the ranks of explicit $4\times4$ and $4\times6$ minors
  of the matrix of $D_{\partial\mathcal K/\partial x}$, all of whose entries are
  of the form $2c_i$ with $c_i$ the coordinates of $x$; since $p\neq2$ the
  factor $2$ is a unit and the same computation works over $K$. One concludes
  as in \cref{prop:klein}, using $p\nmid 3m$.
\end{proof}

\begin{remark}
  The hypothesis $n\geq4$ is used in case (1) of the proof of
  \cite[Theorem 3.5]{MR4771229} and cannot be dropped, see
  \cite[Remark 3.6]{MR4771229}.
\end{remark}

\begin{proposition}[Delsarte hypersurfaces]\label{prop:delsarte}
  Let $n\geq2$ and $d\geq4$ with $(n,d)\neq(2,4)$. If $p\nmid d(d-1)(d-2)$,
  then
  $$\Aut(T^n_d)\cong\ZZ/(d-1)^{n+1}\ZZ.$$
\end{proposition}

\begin{proof}
  By \cref{lem:smooth} the hypersurface is smooth and $\mathcal T$ is
  $p$-tame, so by \cref{lem:tame}
  $$\Vars_p\left(\frac{\partial\mathcal T}{\partial x_i}\right)=
  \begin{cases}
    \{x_0,x_1\} & i=0,\\
    \{x_{i-1},x_i,x_{i+1}\} & 1\leq i\leq n,\\
    \{x_n,x_{n+1}\} & i=n+1,
  \end{cases}$$
  so that $(\beta^*,\leq_{\mathcal T})$ is a poset whose only non-trivial
  relations are $x_0<_{\mathcal T}x_1$ and $x_{n+1}<_{\mathcal T}x_n$. Since
  $\Spar(\mathcal T)=2d-2>4$, \cref{thm:GT} applies: for
  $\varphi\in\Aut(T^n_d)$ there is a permutation $\tau$ with
  $\tau(\{0,n+1\})=\{0,n+1\}$ and $\tau(\{1,\ldots,n\})=\{1,\ldots,n\}$ and
  \begin{equation}\label{eq:delsarte-shape}
    \tvarphi^*x_i=a_ix_{\tau(i)}+b_i\,x_0\,[\tau(i)=1]+c_i\,x_{n+1}\,[\tau(i)=n],
    \qquad a_i\neq0 .
  \end{equation}
  We claim that $\tvarphi$ is diagonal.

  \emph{Step 1: $\tau(0)=0$ and $\tau(n+1)=n+1$.} Otherwise
  $\tvarphi^*x_0=a_0x_{n+1}$ and, by \eqref{eq:chain},
  $\partial\mathcal T/\partial x_0$ and $\partial\mathcal T/\partial x_{n+1}$
  differ by the ring automorphism $\tvarphi^*$ and a nonzero scalar; in
  particular they have the same number of distinct linear factors. But
  $\partial\mathcal T/\partial x_0=(d-1)x_0^{d-2}x_1$ has exactly $2$ of them,
  whereas $\partial\mathcal T/\partial x_{n+1}=x_n^{d-1}+dx_{n+1}^{d-1}$ has
  exactly $d-1\geq3$, because $t^{d-1}+d$ has no repeated root ($p\nmid d(d-1)$).

  \emph{Step 2: $\tau(1)=1$ and $b_1=0$.} By Step 1 we have
  $\tvarphi^*x_0=a_0x_0$, so the only summands of
  $\tvarphi^*(\mathcal T)=\sum_{i=0}^{n}(\tvarphi^*x_i)^{d-1}(\tvarphi^*x_{i+1})
  +(\tvarphi^*x_{n+1})^d$ having a monomial divisible by $x_0^{d-1}$ are those
  with $i=0$ and $i=i_1:=\tau^{-1}(1)$, and they add up to
  $$x_0^{d-1}\left(a_0^{d-1}\,\tvarphi^*x_1+b^{d-1}\,\tvarphi^*x_{i_1+1}\right),
  \qquad b:=b_{i_1}.$$
  Comparing with $\lambda\,\mathcal T$, whose only monomial divisible by
  $x_0^{d-1}$ is $\lambda x_0^{d-1}x_1$, we get that
  $a_0^{d-1}\tvarphi^*x_1+b^{d-1}\tvarphi^*x_{i_1+1}=\lambda x_1$. The terms
  $b_ix_0$ and $c_ix_{n+1}$ in \eqref{eq:delsarte-shape} only contribute to the
  coordinates $x_0$ and $x_{n+1}$, and $1\neq i_1+1$; hence, looking at the
  coordinate $x_{\tau(1)}$, which lies in $\{x_1,\ldots,x_n\}$, we obtain
  $\tau(1)=1$ (otherwise $a_0^{d-1}a_1=0$). Then $i_1=1$ and, looking at the
  coordinate $x_{\tau(2)}\in\{x_2,\ldots,x_n\}$, we obtain $b^{d-1}a_2=0$,
  i.e.\ $b=0$. (Here we use $n\geq2$.)

  \emph{Step 3: induction.} Assume $\tvarphi^*x_k=\alpha_kx_k$ for $0\leq k\leq j$
  with $j\leq n$. For $1\leq j\leq n$ the only $\tvarphi^*x_i$ containing $x_j$
  is $\tvarphi^*x_j$, so the monomials of $\tvarphi^*(\mathcal T)$ divisible by
  $x_j^{d-1}$ are those of $\alpha_j^{d-1}x_j^{d-1}\,\tvarphi^*x_{j+1}$, and
  comparing with $\lambda x_j^{d-1}x_{j+1}$ gives
  $\tvarphi^*x_{j+1}=\alpha_{j+1}x_{j+1}$. Hence $\tvarphi$ is diagonal.

  Finally, $\tvarphi_{ii}^{d-1}\tvarphi_{i+1,i+1}=\tvarphi_{n+1,n+1}^d$ for
  $i\in\{0,\ldots,n\}$; normalizing $\tvarphi_{n+1,n+1}^d=1$, the automorphism is
  determined by $\tvarphi_{00}$, which satisfies
  $\tvarphi_{00}^{d(1-d)^{n+1}}=1$. Since $p\nmid d(d-1)$, the group
  $\mu_{d(d-1)^{n+1}}(K)$ is cyclic of order $d(d-1)^{n+1}$, so we obtain an
  epimorphism $\ZZ/d(d-1)^{n+1}\ZZ\to\Aut(T^n_d)$ with kernel $\mu_d(K)$.
\end{proof}

In each of the three cases above the automorphism group is contained in
$\PGP(V,\beta)$ and splits as a semidirect product of its diagonal part by a
group of permutations of the monomials. We record this decomposition, which is
what the following sections use.

\begin{corollary}\label{cor:DrtimesP}
  Let $X=V(F)\subset\PP(V)$ be a smooth hypersurface of dimension
  $n\geq2$ and degree $d\geq3$ with $(n,d)\neq(2,4)$, all of whose monomials have
  coefficient $1$ and such that $\Aut(X)\subseteq\PGP(V,\beta)$, for instance
  $F^n_d$, $T^n_d$ or $K^n_d$ under the hypotheses of
  \cref{prop:fermat} and \cref{prop:delsarte}. Let $\mathcal{M}_F$ be
  the set of monomials of $F$, let $D\leq\Aut(X)$ be the subgroup of the
  automorphisms admitting a diagonal representative in $\GL(V)$ and let
  $$P:=\{\sigma\in S_{n+2}\mid \sigma(\mathcal{M}_F)=\mathcal{M}_F\}.$$
  Then $D$ is a finite abelian normal subgroup of $\Aut(X)$,
  $$\Aut(X)=D\rtimes P,$$
  and the section $s\colon P\to\Aut(X)$ is given by the permutation matrices
  $\varphi_\sigma$, whose entries lie in $\{0,1\}$.
\end{corollary}

\begin{proof}
  Let $\varphi\in\Aut(X)$ and write $\tvarphi=\tilde\mu\,\tilde P_\sigma$ with
  $\tilde\mu$ diagonal and $P_\sigma$ a permutation matrix; such a decomposition
  is unique and $\varphi\mapsto\sigma$ is a group homomorphism with kernel $D$.
  Since $\tilde\mu^*$ multiplies each monomial by a scalar, we have
  $\tvarphi^*(F)=\sum_{\mathrm{m}\in\mathcal{M}_F}c_{\mathrm{m}}\,\sigma(\mathrm{m})$
  for some $c_{\mathrm{m}}\in K^*$, and $\tvarphi^*(F)=\lambda F$ forces
  $\sigma(\mathcal{M}_F)=\mathcal{M}_F$. As all the coefficients of $F$ are
  equal to $1$, this yields $\tilde P_\sigma^*(F)=F$, so that
  $\varphi_\sigma:=\pi(\tilde P_\sigma)\in\Aut(X)$ and
  $\sigma\mapsto\varphi_\sigma$ is a group-theoretic section of
  $\Aut(X)\to P$. Finally $D$ is abelian, being a group of classes of diagonal
  matrices, and it is finite since $\Aut(X)$ is finite by
  \cref{thm:MM}.
\end{proof}

\begin{remark}\label{rem:standing}
  In the applications to forms over finite fields in the following
  sections we take $k=\FF_q$ with $q$ a power of an odd prime
  $p$ and $K=\overline{\FF}_q$. By
  \cref{prop:fermat}, \cref{prop:delsarte} and
  \cref{cor:DrtimesP}, the decomposition $\Aut(X)\cong D\rtimes P$ with
  $D$ abelian and $P\leq S_{n+2}$ holds, with the same groups $D$ and $P$ as
  over $\CC$, under the hypotheses
  $$p\nmid d(d-1)\ \ \text{(Fermat)},\qquad
  p\nmid d(d-1)(d-2)\ \ \text{(Delsarte)},\qquad
  p\nmid d(d-1)(d-2)m\ \ \text{(Klein)}.$$
\end{remark}

\section{Non-abelian cohomology}\label{sec:cohomology}

In this section, we recall notions of non-abelian cohomology that allow us to study forms of an algebraic variety. For detail see \cite{Se}

\begin{definition}{\label{def:cocycle}}
    Let $\Gamma$ be a profinite group. A $\Gamma$-group $A$ is a $\Gamma$-set that has a group structure invariant under $\Gamma$, thus, $^{s} xy = \ ^{s}x \ ^{s}y$ for all $s \in \Gamma$ and for all $x,y \in A$. We always regard $A$ as a discrete topological space and require the action of $\Gamma$ to be continuous, that is, the stabilizer of every $x \in A$ is an open subgroup of $\Gamma$.
\end{definition}
If $A$ is an abelian group, this definition is equivalent to saying that $A$ is a $\Gamma$-module.
\begin{definition}{\label{def:H1_set}}
    A $1$-cocycle (or simply cocycle) is a continuous map $a: \Gamma \to A$ such that $a_{st} = a_s \ ^{s}a_t$ for all $s,t \in \Gamma$. The set of all cocycles is denoted by $Z^1(\Gamma, A)$.
    We say that two cocycles $a$ and $a'$ are cohomologous or equivalent if and only if there exists $b \in A$ such that $a_s ' = b^{-1} \ a_s \ ^{s}b$. Taking quotients by this equivalence relation we get the first cohomology set of $\Gamma$ on $A$, denoted by $H^1(\Gamma, A)$.
\end{definition}
\begin{remark}
    In general, $H^1(\Gamma, A)$ has no group structure since $A$ is not necessarily abelian. However, it is a pointed set: the class of the trivial cocycle is called the neutral element of $H^1(\Gamma, A)$. This allows us to work with exact sequences of pointed sets.
\end{remark}
Given a $\Gamma$-group $A$ and a cocycle, we can construct a new action of $\Gamma$ on $A$ whose cohomology set is in natural bijection with the original one. This process is called twisting.
\begin{definition}{\label{def:twisting}}
    Let $A$ be a $\Gamma$-group and take a cocycle $a = (a_s) \in Z^1(\Gamma, A)$. The twisted action of $\Gamma$ by $a$ is given by $^{s_{a}}x = a_s \ ^{s} x \ a_{s}^{-1}$.
    We denote by $_aA$ the group $A$ endowed with the twisted action; it is again a $\Gamma$-group. Under this action we define cocycles as in \cref{def:H1_set}, we denote by $Z^1(\Gamma , \ _a A)$ the set of these cocycles and, taking quotients, we write $H^1(\Gamma , \ _{a}A)$ for the first cohomology set (using the twisted action).
\end{definition}
\begin{proposition}{\label{prop:twisting_bijection}}
    Let $a \in Z^1(\Gamma, A)$ and consider $_{a}A$. To each cocycle $a' = (a'_s) \in Z^1(\Gamma, \ _aA)$ we associate the map $s \mapsto a'_s \, a_s$, which is a cocycle of $\Gamma$ in $A$; this defines a bijection
    $$ t_a : Z^1(\Gamma, \ _{a}A) \to Z^1(\Gamma,A) $$
    Taking quotients, $t_a$ defines a bijection
    $$ \tau_a : H^1(\Gamma, \ _{a}A) \to H^1(\Gamma,A)  $$
    which sends the neutral element of $H^1(\Gamma, \ _{a}A)$ into the class of $a$.
\end{proposition}
\begin{proof}
    See \cite[Chapter I, \S5.3, Prop.~35 bis]{Se}.
\end{proof}
\subsection{Exact sequences}
If $A$ and $B$ are $\Gamma$-groups and $u : A \to B$ is a $\Gamma$-equivariant homomorphism, this homomorphism induces a function
$$ v:H^1(\Gamma, A) \to H^1(\Gamma, B) $$
    Let $\alpha \in H^1(\Gamma, A)$ and let $a$ be a cocycle representing it; then $u$ defines a homomorphism $u' :\ _aA \to \ _{u(a)}B$ which induces a function
    $$ v' : H^1(\Gamma, \ _aA) \to H^1(\Gamma,\ _{u(a)}B)$$
    such that the following diagram is commutative
    $$\begin{tikzcd}
	{H^1(\Gamma, A)} & {H^1(\Gamma, B)} \\
	{H^1(\Gamma, \ _a A)} & {H^1(\Gamma , \ _{u(a)}B)}
	\arrow["v", from=1-1, to=1-2]
	\arrow["{\tau_{a}}"', from=2-1, to=1-1]
	\arrow["{v'}", from=2-1, to=2-2]
	\arrow["{\tau_{u(a)}}"', from=2-2, to=1-2]
    \arrow[phantom, "\circlearrowleft", from=1-1, to=2-2]
\end{tikzcd}$$
where $\tau_a$ and $\tau_{u(a)}$ are bijections as in \cref{prop:twisting_bijection}.

In particular, if $A$ is a normal subgroup of $B$, we obtain the following proposition.
\begin{proposition}{\label{prop:action_fibers}}
    Let $A$ and $B$ be $\Gamma$-groups with $A$ normal in $B$, and let $C = B/A$. Then, the sequence of pointed sets
    $$ 1 \to A^{\Gamma} \to B^{\Gamma} \to C^{\Gamma} \xrightarrow[]{\delta} H^1(\Gamma,A) \to H^1(\Gamma,B) \to H^1(\Gamma,C) $$
    is exact. Moreover, 
    \begin{enumerate}
        \item For $c \in C^{\Gamma}$, we compute $\delta(c)$ in the following way: choose a lift $b \in B$ of $c$, then $a_s = b^{-1} \ ^{s}b$ defines a cocycle in $Z^1(\Gamma,A)$ and we take $\delta(c):=[a_s]$.
        \item $C^{\Gamma}$ acts on $H^1(\Gamma, A)$ as follows: If $c \in C^{\Gamma}$ and $b \in B$ is a lift of $c$, for $a = (a_s) \in Z^1(\Gamma, A)$ we set $c \cdot a_s := b^{-1} \ a_s \ ^{s}b$. (This does not depend on the lift of $c$.)
        \item Two classes in $H^1(\Gamma, A)$ have the same image in $H^1(\Gamma, B)$ if and only if they are in the same $C^{\Gamma}$-orbit.
    \end{enumerate}
\end{proposition}
\begin{proof}
    This combines \cite[Chapter I, \S5.5, Prop.38]{Se} (exactness of the sequence), with the description of $\delta$ and some results of \cite[Chapter I, \S5.5, Prop.39]{Se}.
\end{proof}

\begin{corollary}{\label{cor:kernel_as_quotient}}
    The kernel of $H^1(\Gamma ,B) \to H^1(\Gamma, C) $ is identified with the quotient of $H^1(\Gamma ,A)$ by the action of the group $C^{\Gamma}$.
\end{corollary}

Note that if $A$ is normal in $B$, then conjugation by elements of $B$ preserves $A$; in particular, twisting $B$ by a cocycle $b \in Z^1(\Gamma, B)$ induces a twisted action of $\Gamma$ on $A$.
\begin{proposition}{\label{prop:exact_sequence_twisted}}
    Let $A$ and $B$ be $\Gamma$-groups with $A$ normal in $B$, let $C = B/A$ with canonical projection $\pi \colon B \to C$, and let $b = (b_s)$ be a cocycle in $B$. Then the sequence of $\Gamma$-groups
    $$ 1 \to\ _{b}A \to {_b}B \to _{\pi(b)}C \to 1 $$
    is exact.
\end{proposition}
\begin{proof}
    For $\sigma \in \Gamma$ the twisted action is $^{\sigma_b} a = b_\sigma \ ^{\sigma}a \ b_{\sigma}^{-1} $ for all $a \in A$, and since $A$ is normal we have $b_{\sigma}Ab_{\sigma}^{-1} = A$. And the inclusion $i :_{b}A \hookrightarrow \ _{b}B$ is $\Gamma$-equivariant since the twist in $A$ is the restriction of the twist on $B$, i.e., $i(^{\sigma_{b}}a) = \ ^{\sigma_b}a = ^{\sigma_b}(i(a)) $. In the same way, $\pi : \ _{b}B \to _{\pi(b)}C$ is $\Gamma$-equivariant, since for $b' \in B$, $\pi (^{\sigma _b} b') =\pi ( b_{\sigma} \ ^{\sigma}b' \ b_\sigma ^{-1} ) =  \pi( b_{\sigma})  \pi( ^{\sigma}b') \pi (b_\sigma) ^{-1} = c_{\sigma} \pi (b') c_{\sigma}^{-1} = ^{\sigma_{\pi(b)}} \pi(b') $. Finally, by hypothesis we have that $i$ is injective, $\pi$ surjective and $\operatorname{Im}(i)=\ker(\pi)$, hence the sequence is exact.
\end{proof}
\begin{corollary}{\label{cor:twisted_kernel_quotient}}
    Let $A$ and $B$ be $\Gamma$-groups with $A$ normal in $B$, let $C = B/A$, and let $b = (b_s)$ be a cocycle in $B$. Then the sequence of pointed sets
    $$ 1 \to (_{b}A)^{\Gamma} \to (_{b}B)^{\Gamma} \to ( _{\pi(b)}C)^{\Gamma} \xrightarrow[]{\delta'} H^1(\Gamma, \ _{b}A) \to H^1(\Gamma,\ _{b}B) \to H^1(\Gamma, \ _{\pi(b)}C) $$
    is exact. Moreover, the kernel of $H^1(\Gamma, \ _{b}B) \to H^1(\Gamma,\ _{\pi(b)}C) $ is identified with $H^1(\Gamma, \ _{b}A)/(_{\pi(b)}C)^{\Gamma}$.
\end{corollary}
\begin{proof}
    It is a consequence of  \cref{prop:exact_sequence_twisted} using \cref{prop:action_fibers} and  \cref{cor:kernel_as_quotient}. 
\end{proof}

\subsection{Forms of an algebraic variety}

We now apply this machinery to the study of forms of algebraic varieties. First of all, let $k$ be a field and $K$ a Galois extension of $k$.
\begin{definition}
    Let $X$ be an algebraic variety over $K$. A $k$-form of $X$ is an algebraic variety $X'$ defined over $k$ such that $(X')_{K} := X' \times_{\Spec(k)}\Spec(K)$ is $K$-isomorphic to $X$. The set of $k$-forms of $X$ modulo $k$-isomorphism is denoted by $E(K/k, X)$.
\end{definition}
\begin{example}{\label{ex:k_forms}}
We show the notion of a $k$-form with the following examples:
    \begin{enumerate}
        \item $\PP^{n}_{k}$ is a $k$-form of $\PP^n_K$;
        \item $X_1=\PP^{1}_{\RR}$ and $X_2 =\{[x:y:z] \in \PP^2_{\RR} : x^2 + y^2 + z^2 = 0\}$ are two $\RR$-forms (or real forms) of $\PP^1_{\CC}$. In fact, $(X_2)_{\CC}$ is a smooth conic in $\PP^2_{\CC}$, then isomorphic to $\PP^1_{\CC}$, and they are not isomorphic over $\RR$, since $X_1(\RR) \neq \emptyset$ while $X_2 (\RR) = \emptyset$;
        \item The curves $C_1 = \{(x,y) \in \mathbb{A}^2_{\RR} : y^2 - x^3 - x=0\}$ and $C_2 = \{(x,y) \in \mathbb{A}^2_{\RR} : y^2 - x^3 +x=0\}$ are $\RR$-forms of the same complex curve. Moreover, the isomorphism over $\CC$ is defined as
        \begin{align*}
            \varphi: (C_1)_{\CC} &\to (C_2)_{\CC}\\
            (x,y) &\mapsto (-ix , e^{i \frac{\pi}{4}}y)
        \end{align*} 
    \end{enumerate}
\end{example}

As illustrated in \cref{ex:k_forms}, a variety may admit more than one $k$-form up to isomorphism. In general the set $E(K/k,X)$ need not even be finite, as recalled in the introduction. However, for a smooth hypersurface as in \cref{thm:MM}, $\Aut(X)$ is already finite, so $E(K/k,X)$ is as well finite by \cref{prop:bijection_forms_cohomology}. In the rest of the paper we compute this finite number exactly for the Fermat, Klein and Delsarte hypersurface.

Since $\Gal(K/k)$ is a profinite group, it is natural to consider, once a $k$-form $X'$ of $X$ is fixed and $X$ is identified with $(X')_K$, for $\sigma \in \Gal(K/k)$ the action on $\Aut_{K}(X)$ given by $^{\sigma} \phi := (1 \times \sigma) \ \phi \ (1 \times \sigma^{-1} )$ and, in a similar way to \cref{def:H1_set}, construct $H^1(\Gal(K/k), \Aut_K (X))$.
\begin{proposition}{\label{prop:bijection_forms_cohomology}}
    We get an injective function $\theta : E(K/k, X) \to H^1(\Gal(K/k), \Aut_K (X))$.
    If $X$ is quasiprojective, then $\theta$ is bijective.
\end{proposition}
\begin{proof}
    Injectivity is elementary; surjectivity uses Weil's descent method. See \cite[Chapter III, \S1.3, Prop.~5]{Se}.
\end{proof}
Thus, the problem of counting the $k$-forms of an algebraic variety is reduced to understanding this cohomology set, as we will do in the next section.

\section{Computing the forms of classical hypersurfaces}\label{sec:forms}

As we proved in \cref{sec:diff}, there exist hypersurfaces $X$ such that $\Aut(X)= D \rtimes P$, where $D$ is an abelian group of diagonal automorphisms and $P$ is the subgroup of automorphisms corresponding to permutations of monomials, isomorphic to a subgroup of $S_{n+2}$. In this section, we use this decomposition to describe $H^1(\Gal(K/k), \Aut_{K}(X) )$ for this family of hypersurfaces in the cases $k \in \{\RR , \FF_q\}$, where $q$ is a power of some prime $p >2$, and $K$ the algebraic closure of these fields.

In what follows, we denote the profinite group $\Gal(K/k)$ by $\Gamma$ and $\Aut_K(X)$ simply by $\Aut(X)$.

\begin{proposition}{\label{prop:equivariant_section}}
    Let $X = (X')_K \subset \PP^{n+1}_{K}$ be a hypersurface all of whose monomials have coefficient $1$ and such that (cf.\ \cref{cor:DrtimesP}) $\Aut(X) \cong D \rtimes P$ with $D$ diagonal and $P\leq S_{n+2}$ the group of permutations of monomials of $F$. Let $\beta$ be a basis of $V$ defined over $k$ and $\Gamma$ acting on $\Aut(X)$ as $^{s} \phi := (1 \times s) \ \phi \ (1 \times s^{-1} )$. Then
    $\Gamma$ acts trivially on $P$ and the section $s:P\hookrightarrow \Aut(X)$ given by $s(p) = \varphi_p$ is $\Gamma$-equivariant.
\end{proposition}
\begin{proof}
    By \cref{cor:DrtimesP} $\varphi_p = (\varphi_{ij})$ and $^{\sigma}\varphi = (^{\sigma}\varphi_{ij} )$, but since $\varphi_p$ has entries with values in $\{0,1\} \subset k$ and $k$ is $\Gamma$-invariant we obtain $^{\sigma}\varphi_p = \varphi_p$, i.e., $^{\sigma}s(p)=s(p)$. Now, for the equivariance of $s$ note that if $\mu \in D$ then $^{\sigma}\mu$ is diagonal and $^{\sigma}\mu \in D$, using that $D$ is normal in $\Aut(X)$ we obtain that $\Gamma$ act on $\Aut(X)/D \cong P$, which implies that  $\pi:\Aut(X) \to P$ is $\Gamma$-equivariant. Then, $^{\sigma}p = \pi(^{\sigma} s(p)) = \pi(s(p)) = p$ gives us $^{\sigma}s(p) = s(^{\sigma}p)$ for all $p \in P$, which shows that $s$ is $\Gamma$-equivariant.
\end{proof}

\begin{theorem}[Theorem A]\label{thm:A}
Let $X=V(F)\subset\PP^{n+1}_{K}$ be a smooth hypersurface of dimension $n\geq1$ and degree $d\geq3$, with $(n,d)\notin\{(1,3),(2,4)\}$. Let $\Aut(X)\cong D\rtimes P$ with $D$ the subgroup of diagonal automorphisms and $P\leq S_{n+2}$ the group of permutations of the monomials of $F$, all of which have coefficient $1$ (cf.\ \cref{cor:DrtimesP}), with $X$ defined over $k$ and $\Gamma$ acting as in \cref{prop:equivariant_section}. Then
$$ |H^1(\Gamma,\Aut(X))| = \sum_{[c]\in H^1(\Gamma,P)} |H^1(\Gamma,\ _{c}D)/C_P(c)|. $$
\end{theorem}
\begin{proof}
    By \cref{prop:equivariant_section} the action of $\Gamma$ on $P$ is trivial and $s$ is $\Gamma$-equivariant. Hence the projection $\pi : \Aut(X) \to P$ induces a surjective function $\pi_* : H^1(\Gamma, \Aut(X)) \to H^1(\Gamma, P)$ which admits a $\Gamma$-equivariant section $s_*$ induced by $s$. Thus, $H^1(\Gamma, \Aut(X))$ is the disjoint union of $\pi_* ^{-1}([c])$, with $[c] \in H^1(\Gamma, P)$. The section $s_*$ allows us to lift $[c] \in H^1(\Gamma, P)$ to $s_*([c]) \in H^1(\Gamma, \Aut(X))$, therefore, we can twist $H^1(\Gamma, P)$ and $H^1(\Gamma  , \Aut(X))$ by $c$ and $s(c)$, respectively. By \cref{prop:twisting_bijection} we have the bijections $\tau_{s(c)}$ and $\tau_{c}$ explained in the following diagram
    $$\begin{tikzcd}
	   {H^1(\Gamma, D)} & {H^1(\Gamma, \Aut(X))} & {H^1(\Gamma, P)} \\
	   {H^1(\Gamma, \ _{c}D)} & {H^1(\Gamma, \ _{s(c)} \Aut(X))} & {H^1(\Gamma, \ _{c}P)}
	   \arrow[from=1-1, to=1-2]
	   \arrow["\pi_{*}", from=1-2, to=1-3]
	   \arrow["{s_*}"', shift left, curve={height=18pt}, from=1-3, to=1-2]
	   \arrow[from=2-1, to=2-2]
	   \arrow["{\tau_{s(c)}}"', from=2-2, to=1-2]
	   \arrow["{\pi_{*}'}", from=2-2, to=2-3]
	   \arrow["{\tau_{c}}"', from=2-3, to=1-3]
        \arrow[phantom, "\circlearrowleft", from=1-2, to=2-3]
    \end{tikzcd}$$
    Applying \cref{cor:twisted_kernel_quotient} to $A=D$, $B=\Aut(X)$, $C=P$ and the cocycle $b=s(c)$ in $\Aut(X)$ with $\pi(s(c))=c$, we obtain
    $$ \pi_{*}^{-1}([c]) \cong \ker\big(H^1(\Gamma, \ _{s(c)}\Aut(X)) \to H^1 (\Gamma, \ _{c}P)\big)  \cong H^1(\Gamma, \ _{c}D)/(\ _{c}P)^{\Gamma} $$
    Since $\Gamma$ acts trivially on $P$, we have $^{\sigma_{c}}  p = c \ ^{\sigma}p  \ c^{-1} = c \ p \ c^{-1}$ for $p \in {}_{c}P$, so $(\ _{c}P)^{\Gamma} = \{ p \in \ _{c}P : c \ p \ c^{-1} = p\} = C_P (c)$.
    Thus, it follows that $|H^1(\Gamma, \Aut(X)) | = \displaystyle \sum_{[c] \in H^1(\Gamma , P )} |H^1(\Gamma, \ _{c}D) / C_P (c) |$, where $p \in C_P(c)$ acts on $H^1(\Gamma, \ _{c}D)$ by $p \cdot [\mu] = [s(p)^{-1} \  \mu_\sigma \ ^{\sigma_{s(c)}} s(p)]$, using the lift $s(p)$ of $p$.
\end{proof}

\subsection{\texorpdfstring{$\RR$-forms}{R-forms}}

In this context, $k=\RR$ and $K=\CC$, so an $\RR$-form (or real form) of $X$ is an algebraic variety over $\RR$ whose complexification is $\CC$-isomorphic to $X$.
Let $\sigma$ be the complex conjugation in $\CC$; if $X$ is defined over $\RR$, i.e., $X=(X')_{\CC}$, then $1 \times \sigma$ is an antiregular involution of $X$ known as a real structure.
\begin{definition}
    A real structure $\rho$ on $X$ is an antiregular involution $\rho : X \to X$, that is, a morphism of schemes such that $\rho^2 = \Id$ and that makes the following diagram commutative
    $$\begin{tikzcd}
	   X & X \\
	   {\Spec(\CC)} & {\Spec(\CC)}
	   \arrow["\rho", from=1-1, to=1-2]
	   \arrow[from=1-1, to=2-1]
	   \arrow[from=1-2, to=2-2]
	   \arrow["{\Spec(z \mapsto \overline{z})}", from=2-1, to=2-2]
    \end{tikzcd}$$
    We say that two real structures on $X$ are equivalent if and only if there exists $\phi \in \Aut(X)$ such that $\rho' =  \phi^{-1} \rho \phi$. The real locus of $\rho$ is the set $X^{\rho}$ of fixed points of $\rho$.
\end{definition}
\begin{example}
    The complex conjugation gives us the following examples:
    \begin{enumerate}
        \item On $\PP^1_{\CC}$ we can define two real structures
        $$\begin{aligned}
            \tilde \rho : \PP^{1}_{\CC} &\to \PP^{1}_{\CC} \\
            [x : y] &\mapsto [\overline{x} : \overline{y}]
        \end{aligned}
        \qquad\text{and}\qquad
        \begin{aligned}
            \tilde \rho' : \PP^{1}_{\CC} &\to \PP^{1}_{\CC} \\
            [x:y] &\mapsto [- \overline{y} : \overline{x}]
        \end{aligned}$$
        They are not equivalent since $(\PP^1_{\CC})^{\tilde \rho} \neq \emptyset$ and $(\PP^1_{\CC})^{\tilde \rho'} = \emptyset$. Note that this is similar to the case of real forms of $\PP^1_{\CC}$ shown in \cref{ex:k_forms}.
        \item We can generalize to $\PP^{n+1}_{\CC}$ the real structure defined above on $\PP^{1}_{\CC}$ 
        \begin{align*}
            \tilde \rho : \PP^{n+1}_{\CC} &\to \PP^{n+1}_{\CC} \\
            [x_0 : \cdots : x_{n+1}] &\mapsto [\overline{x}_0 : \cdots : \overline{x}_{n+1}]
        \end{align*}
        \item If $n$ is even, then we have an additional real structure on $\PP^{n+1}_{\CC}$, given by
        \begin{align*}
            \tilde \rho' : \PP^{n+1}_{\CC} &\to \PP^{n+1}_{\CC} \\
            [x_0 : x_1: \cdots : x_n: x_{n+1}] &\mapsto [ -\overline{x}_1 : \overline{x}_0  : \cdots :  - \overline{x}_{n+1} : \overline{x}_{n}]
        \end{align*}
    \end{enumerate}
\end{example}
In general, we can take the real structure $\tilde \rho$ on $\PP^{n+1}_{\CC}$ and restrict it to $X$ to obtain a real structure $\rho$ on $X$.
With this language, we can state \cref{prop:bijection_forms_cohomology} as the following theorem
\begin{theorem}[\cite{BoSe64}] 
    Let $X$ be a complex quasiprojective variety with a real structure $\rho$. There is a natural bijection between the set of $\RR$-isomorphism classes of real forms of $X$ and $H^{1} (\Gamma , \Aut_{\CC}(X))$, where the non-trivial element $\sigma$ of $\Gamma = \{1 , \sigma\}$ acts on $\Aut_{\CC}(X)$ by $^\sigma \varphi = \rho  \varphi \rho^{-1}$.
\end{theorem}
Since $\Aut(X) \cong D \rtimes P$, the homomorphisms of $\Gamma$-groups $D \hookrightarrow \Aut(X)$ and $\Aut(X) \to P$ allow us to use the induced functions to study $H^1(\Gamma, \Aut(X))$ in terms of $H^1(\Gamma, D)$ and $H^1(\Gamma, P)$, as in \cref{thm:A}.

\begin{lemma}{\label{prop:action_of_Gamma}}
    Let $X = V(F) \subset \PP^{n+1}_{\CC}$ be a complex hypersurface such that $\Aut(X) \cong D \rtimes P $, and $\rho $ the real structure on $X$ induced by complex conjugation. If $\mu = \diag \left( \mu_{0} , \ldots, \mu_{n+1} \right) \in D$ (whose entries may be taken to be roots of unity, since $D$ is finite) and $p \in P$, then
    \begin{enumerate}
        \item $^{\sigma} \mu= \rho \mu \rho  = \mu^{-1}$;
        \item $^{\sigma} p= \rho p \rho = p$;
        \item $p \ \mu \ p^{-1} = \diag \left(\mu_{p^{-1}(0)} , \ldots, \mu_{p^{-1}(n+1)} \right) $
    \end{enumerate}
\end{lemma}
\begin{proof}
    We see the action on homogeneous coordinates:
    \begin{enumerate}
        \item Since $\mu_j$ are roots of unity, then $\overline{\mu_j} = \mu_j^{-1}$. Thus,
        \begin{align*}
            \rho (\mu (\rho \cdot [x_0 :  \cdots :x_{n+1}])) &= \rho (\mu \cdot[\overline{x_0} : \cdots : \overline{x_{n+1}}] ) \\
            &= \rho \cdot [\mu_0 \overline{x_0} : \cdots :\mu_{n+1} \overline{x_{n+1}}] \\
            &= [\mu_0^{-1} x_0  : \cdots : \mu_{n+1}^{-1} x_{n+1}]
        \end{align*}
        This is equivalent to act by $\mu^{-1}$.
        \item This is a consequence of \cref{prop:equivariant_section} with $k=\RR$.
        \item We see that  
        \begin{align*}
            p \mu ( p^{-1} \cdot [x_0 : \cdots : x_{n+1}] ) &= p \mu \cdot [x_{p(0)} : \cdots : x_{p(n+1)}] \\
            &= p \cdot [\mu_0 x_{p(0)} : \cdots : \mu_{n+1} x_{p(n+1)}] \\
            &=[\mu_{p^{-1}(0)} x_{0} : \cdots : \mu_{p^{-1}(n+1)} x_{n+1}]
        \end{align*}
        This is equivalent to acting by the automorphism $\diag \left(\mu_{p^{-1}(0)} , \ldots, \mu_{p^{-1}(n+1)} \right) \in D$.
    \end{enumerate}
\end{proof}

\begin{proposition}{\label{prop:H1(D)_descomposition}}
    The set $H^1(\Gamma,D)$ is identified with $(\ZZ/2\ZZ)^r$, for some $r \in \NN$.
\end{proposition}
\begin{proof}
    By \cref{prop:action_of_Gamma}, $\Gamma$ acts on $D$ by inversion. Thus, the cocycle condition is $a_{\sigma} a_{\sigma}^{-1} = 1$. As this holds for all $a_{\sigma} \in D$ we get that $Z^1(\Gamma, D)$ is identified with $D$. Now, $a'$ and $a$ are equivalent if and only if $a' = b^{-2}a$ for some $b \in D$. Thus, $H^1(\Gamma, D) \cong D/D^2$. Using the fact that $D$ is abelian and finite, then by the classification of finite abelian groups we have that $D \cong \displaystyle \bigoplus_{i=1}^{k} \ZZ / n_i\ZZ$ such that $n_{1}|n_{2} | \cdots |n_{k}$, then, 
    $$D^2 =\bigoplus_{i=1}^{k} 2( \ZZ/n_i\ZZ) \cong \bigoplus_{i=1}^{k} (\ZZ/n_i\ZZ) / (\ZZ /\gcd(2 , n_i)\ZZ) \cong \bigoplus_{i=1}^{k} \left(\ZZ /\frac{n_i}{\gcd(2,n_i)}\ZZ \right)$$ 
    Then, $H^1(\Gamma ,D) \cong D/D^2 \cong \displaystyle \bigoplus_{i=1}^{k} \ZZ / \gcd(2,n_i)\ZZ \cong  (\ZZ/2\ZZ)^{r}$ where $r$ is the number of even invariant factors $n_i$.
\end{proof}

\begin{proposition}{\label{prop:H1(P)_involutions}}
    The set $H^1(\Gamma,P)$ is identified with the set of elements of order at most $2$ of $P$ modulo conjugation.
\end{proposition}
\begin{proof}
    By \cref{prop:action_of_Gamma}, $\Gamma$ acts trivially on $P$, then the cocycle condition is $a_\sigma^{2} = 1$. Thus, the elements of $Z^1(\Gamma, P ) $ correspond to elements of order at most $2$ in $P$. Now, two cocycles are equivalent if and only if they are conjugate by an element of $P$. This shows that $H^1(\Gamma , P)$ corresponds to elements of order at most $2$ of $P$ up to conjugation in $P$; for $P=S_{n+2}$, two such elements are conjugate if and only if they have the same cycle type.
\end{proof}
In this way, \cref{thm:A} is equivalent to 
\begin{theorem}\label{theorem A}
    Let $X = V(F) \subset \PP^{n+1}_{\CC}$ be a complex hypersurface of dimension $n\geq 1$ and degree $d \geq 3$, with $(n,d) \not \in \{(1,3),(2,4)\}$. Let $\Aut(X) \cong D \rtimes P$, with $D$ abelian and $P\leq S_{n+2}$. Then, the number of $\RR$-forms of $X$ is given by
    $$ |H^1(\Gamma, \Aut(X))| = \sum_{[p_l] \in H^1(\Gamma , P )} |H^1(\Gamma, \ _{p_l}D) / C_P (p_l) |$$
    where $C_P (p_l)$ is the centralizer of the element $p_l$ of order at most $2$ with $l$ fixed points, and the set  $H^1(\Gamma, \ _{p_l}D) / C_P (p_l) $ contains of orbits of the action of $C_P(p_l)$ on $H^1(\Gamma, \ _{p_l}D)$.
\end{theorem}
Moreover, we can give more information about the twisted cohomology set of $\Gamma$ in $D$ by $p_l$.

\begin{theorem}[Theorem B]{\label{thm:B}}
    Let $p_l$ be an involution in $P$, then $|H^1(\Gamma,\ _{p_l}D)| = \dfrac{|\ker (1 - {p_l} _{*})|}{|\operatorname{Im}(1 + {p_l}_{*})|}$, where ${p_{l}}_{*}$ denotes the conjugation by $p_l$ on $D$.
\end{theorem}
\begin{proof}
    The condition $a \in Z^1 (\Gamma, \ _{p_l}D)$ is equivalent to $a_{\sigma} \ p_l a_{\sigma}^{-1} p_{l} ^{-1} =1$, which is equivalent to $a_{\sigma} = p_l \ a_{\sigma} p_l^{-1}$. If we consider ${p_l}_* : D \to D$ given by ${p_l}_* (d) := p_l d p_l ^{-1}$, we obtain $Z^1(\Gamma, \ _{p_l}D) = \ker(1-{p_l}_*) = \{ \mu \in D : \mu(p_l \ \mu \ p_{l}^{-1})^{-1} = 1 \} = \{\mu \in D: \mu = p_l \ \mu \ p_l \} = C_D(p_l)$. Furthermore, two cocycles $a$ and $a'$ are equivalent if and only if there exists $b \in D$ such that $a'_{\sigma} = b^{-1} \ a_{\sigma} \ ^{\sigma_{p_l}}(b) = b^{-1} \ a_{\sigma} \ p_l \ b^{-1} \ p_{l}^{-1} $. Now, let $1 + {p_l}_*$ be given by $(1 + {p_l}_*)(b) = b \cdot p_l\ b \ p_l^{-1} $, in these terms, the equivalence of cocycles is $a'_{\sigma} \cdot a_{\sigma}^{-1} \in \operatorname{Im}(1 + {p_l}_*)$. In consequence, $ H^1 (\Gamma, \ _{p_l}D) = Z^1(\Gamma , \ _{p_l}D)/\sim = \ker(1 - {p_l}_*)/\operatorname{Im}(1 + {p_l}_*) $.
\end{proof}
We use these theorems to study the real forms of hypersurfaces whose automorphism group is of the form $D \rtimes P$. By \cref{cor:DrtimesP}, \cref{prop:fermat} and \cref{prop:delsarte}, the Fermat, Delsarte and Klein hypersurfaces satisfy this condition, and we now count their real forms.

\subsubsection{Count of real forms of classical hypersurfaces}
Let $T^n_d :=\{x_0^{d-1}x_1 + \cdots + x_{n+1}^{d} = 0\} \subset \PP^{n+1}_{\CC}$ be the Delsarte hypersurface of dimension $n$ and degree $d$, then \cite{MR4771229} says that for $n\geq 2$ and $d\geq 4$ with $(n,d) \neq (2,4)$ we get $\Aut(T^n_d)\cong \ZZ/(d-1)^{n+1}\ZZ $.
\begin{corollary}{\label{cor:R_forms_Delsarte}}
    Let $T^n_d \subset \PP^{n+1}_{\CC}$ be the Delsarte hypersurface of dimension $n \geq 2$ and degree $d \geq 4$ with $(n,d)\neq(2,4)$. Then, the number of real forms is
    $$ |H^1(\Gamma,\Aut(T^n_d))| = \begin{cases} 1 & \text{if $d$ is even,}\\ 2 & \text{if $d$ is odd} \end{cases} $$
\end{corollary}
\begin{proof}
    By \cref{prop:delsarte} we obtain $P = \{1\}$, then $H^1(\Gamma, P) = \{1\}$. By \cref{prop:H1(D)_descomposition} we see $H^1 (\Gamma, D) \cong (\ZZ/2\ZZ)^{r}$, where $r \in \{0,1\}$. Let $N = (d-1)^{n+1}$, then $D \cong \ZZ/N\ZZ$ is cyclic and we conclude that $r=1$ if $N$ is even and $r= 0$ if $N$ is odd. Since $N = (d-1)^{n+1}$ is even if and only if $d$ is odd, we conclude the result.
\end{proof}

\begin{remark}
    It is well known in the literature that if $\ZZ/2\ZZ$ acts trivially or by inversion on some cyclic group $\ZZ/N\ZZ$, then $H^1(\ZZ/2\ZZ , \ZZ/N\ZZ)$ is trivial or isomorphic to $\ZZ/2\ZZ$ depending on the parity of $N$. Thus, this case (where $P = \{1\}$) can be solved computing cohomology in the usual way. However, for the hypersurfaces described below, in which $P \neq \{1\}$ we will twist by cocycles as in the theorem.
\end{remark}

Let $K^{n}_d := \{x_0^{d-1}x_1 + \cdots + x_{n+1}^{d-1}x_0 = 0\} \subset \PP^{n+1}_{\CC}$  be the Klein hypersurface of dimension $n\geq 2$ and degree $d\geq 4$ with $(n,d)\neq (2,4)$. By \cite{MR4771229}, $\Aut(K^n_d) \cong \ZZ/m\ZZ \rtimes \ZZ/(n+2)\ZZ$, where $m=\frac{(d-1)^{n+2} - (-1)^{n+2}}{d}$. In particular, doing some computations, we note the following proposition about the parity of $m$ in terms of $n$ and $d$.
\begin{proposition}{\label{prop:parity_of_m}}
    Let $m = \dfrac{(d-1)^{n+2} - (-1)^{n+2}}{d}$. Then, $m$ is even if and only if $n$ and $d$ are even.
\end{proposition}
\begin{proof}
    Write $x=d-1$, so that $dm=x^{n+2}-(-1)^{n+2}$ and $d=x+1$. The
    geometric series (telescoping against the factor $x+1$) gives
    $$ m=\sum_{j=0}^{n+1}(-1)^{n+1-j}\,x^{j}
      =x^{n+1}-x^{n}+\cdots+(-1)^{n+1}. $$
    If $d$ is odd, then $x$ is even, every summand with $j\geq1$ is even and
    $m\equiv(-1)^{n+1}\equiv1\pmod 2$, so $m$ is odd. If $d$ is even, then
    $x$ is odd, every one of the $n+2$ summands is odd, and hence
    $m\equiv n+2\equiv n\pmod2$. Therefore $m$ is even if and only if $d$
    and $n$ are both even.
\end{proof}
This proposition allows us to discard a case in the study of the cohomology in terms of the parity of $n$ and $d$ in the following corollary.
\begin{corollary}\label{cor:R_forms_Klein}
    Let $K^n_d \subset \PP^{n+1}_{\CC}$ be the Klein hypersurface of dimension $n\geq 2$ and degree $d\geq 4$ with $(n,d)\neq(2,4)$. Then, the number of real forms is 
    $$ |H^1(\Gamma,\Aut(K^n_d))| = \begin{cases}
        1 & \text{if $n$ is odd,}\\
        2 & \text{if $n$ is even and $d$ odd,}\\
        4 & \text{if $n$ and $d$ are even.}
    \end{cases}$$
\end{corollary}
\begin{proof}
    In this case $D = \langle \mu \rangle  \cong \ZZ/m\ZZ$ and $P = \langle \tau \rangle \cong \ZZ/(n+2)\ZZ$. Since $P$ is a cyclic group and $\Gamma$ acts trivially on $P$ (by \cref{prop:H1(P)_involutions}) we obtain
    $$ H^1 (\Gamma , P) = \begin{cases}
    \{[1]\} & \text{if $n$ is odd}; \\
    \{ [1] , [\tau_0]\} & \text{if $n$ is even}.
    \end{cases} $$
    where $\tau _0 = \tau^{\frac{n+2}{2}}$ is the unique element of order $2$ in $P$.
    Now, we study fibers of $H^1(\Gamma, P)$:
    \begin{enumerate}
        \item Fiber of $[1]$: By \cref{prop:H1(D)_descomposition} to $D = \ZZ/m\ZZ$ we obtain $H^1(\Gamma, D) \cong (\ZZ/2\ZZ) ^r$, with $r =1$ if $m$ is even and $r=0$ if $m$ is odd. Using \cref{cor:kernel_as_quotient} we determine when two elements of $H^1(\Gamma , D)$ lie in the same image in $H^1(\Gamma, \Aut(K^n_d))$: we study the action of $C_P(1) = P = \langle \tau \rangle$, given by $\tau^j ([\mu^k]) = [\tau^{-j} \ \mu^{k} \ \tau^j]= [\mu^{k(d-1)^{j}}]$. Taking $\mu^k = 1$, we see that $\tau^j ([1]) = [1]$ for all $\tau^j \in \langle \tau \rangle =P$. This means that $[1]$ defines a single orbit (which does not depend on the parity of $m$). Since $H^1(\Gamma, D)$ has at most $2$ elements in different orbits we obtain $H^1(\Gamma, D) \hookrightarrow H^1(\Gamma, \Aut(K^n_d))$, then, the fiber of $[1]$ has $1$ element if $m$ is odd, or $2$ elements if $m$ is even.
        \item Fiber of $[\tau_0]$: Assuming that $n$ is even, we see that the action of the only element of order $2$ $p_l \in P$ on $D$ is $^{\sigma_{\tau_0}}(\mu) = \tau_0 \ ^{\sigma}\mu \ \tau_0 ^{-1} = \mu^{-(1-d)^{\frac{n+2}{2}}} $. Let $g =(1-d)^{\frac{n+2}{2}} $ (note that $g$ may be negative; all the gcd's below are taken non-negative), then $(1-d)^{n+2} \equiv 1 \mod m $ and we see that $g^2 \equiv 1 \mod m$, furthermore, $g \equiv 1 \mod d$ and \cref{thm:B} says
        \begin{align*}
            |H^1(\Gamma ,\ _{\tau_0} D)| = \dfrac{|\ker(1-{\tau_0}_{*})|}{|\operatorname{Im}(1+ {\tau_0}_{*})|}  = \dfrac{\gcd(m,1-g)}{\left( \frac{m}{\gcd(m,1+g)}  \right)} = \dfrac{\gcd(m,1-g) \gcd(m,1+g)}{m}
        \end{align*}
        Again, $(g-1)(g+1) = dm$ and using the fact that $g\equiv 1 \mod d$ we get $g-1 = dq$ for some $q \in \ZZ$, then, the equality is equivalent to $dq (g+1) = dm $ which imply $q(g+1) = m$ and we obtain $\gcd(g+1,m) = g+1$. Since $g-1 = dq$ and $m = q(g+1)$, then 
        $$ \gcd(g-1,m) = \gcd(dq , q(g+1)) = |q| \gcd (d,g+1) $$
        Now, as $g\equiv 1 \mod d$, then $g+1 \equiv 2 \mod d$ and we see $\gcd(d,g+1) = \gcd(d,2)$, in this way
        $$ |H^1(\Gamma, \ _{\tau_0}D)| = \dfrac{|q| (g+1) \gcd(d,2)}{m} = \gcd(d,2) $$
    
        This is equivalent to say that $H^1 (\Gamma, \ _{\tau_0}D) \cong (\ZZ/2\ZZ)^r$, with $r \in \{0,1\}$ depending on the parity of $d$. Finally, the action of $C_P(\tau_0) = P$ on $H^1(\Gamma, \ _{\tau_0}D)$ is $\tau^{j} \cdot (\mu^k) = \tau^{-j} \ \mu^{k} \ \tau^j = \mu^{k(d-1)^{j}} $, but this is similar to the case of the fiber of $[1]$ shown above, then a similar argument proves that $C_P(\tau_0)= \langle\tau \rangle$ acts trivially on $H^1(\Gamma, \ _{\tau_0}D)$, then the fiber of  $[\tau_0]$ has $1$ element if $d$ is odd, or $2$ elements if $d$ is even.
    \end{enumerate}
    Note that, if $n$ is even, by \cref{prop:parity_of_m} we obtain $\gcd(d,2) = \gcd(m,2)$, and as a consequence, we can write $|H^1(\Gamma, \ _{\tau_0}D)|=\gcd(2,m)$.
    Adding over the fibers, \cref{thm:A} gives the result.
\end{proof}
In both cases $T^n_d$ and $K^n_d$ we used the fact that the factors of the semidirect product are cyclic groups, which makes the description easier since we can express the elements in terms of a single generator; this is no longer the case for the Fermat hypersurface, since $P=S_{n+2}$. However, the formula gives another way to prove the theorem of \cite{Sasaki}, which counts the real forms of Fermat hypersurfaces.

Let $F^{n}_d := \{x_0^d + x_1^d + \cdots + x_{n+1}^{d} = 0\} \subset \PP^{n+1}_{\CC}$  be the Fermat hypersurface of dimension $n$ and degree $d$. By \cite{MR4771229} we know that if $(n,d)\notin \{(1,3), (2,4)\}$ then $\Aut(F^n_d) \cong (\ZZ/d\ZZ)^{n+1} \rtimes S_{n+2}$.

\begin{corollary}[\cite{Sasaki}]\label{cor:R_forms_Fermat}
    Let $F^n_d \subset \PP^{n+1}_{\CC}$ be the Fermat hypersurface of dimension $n$ and degree $d$ with $(n,d) \notin \{(1,3),(2,4)\}$. Then, the number of real forms is
    $$ |H^1(\Gamma, \Aut(F^n_d))| = \begin{cases}
        \left\lfloor \dfrac{n}{2} \right\rfloor+ 2 & \text{if $d$ is odd,}\\
        \dfrac{(n+3)(n+5)}{8} & \text{if $d$ is even and $n$ odd,}\\
        1 + \dfrac{(n+4)(n+6)}{8} & \text{if $d$ and $n$ are even.}
    \end{cases} $$
\end{corollary}
\begin{proof}
    For this hypersurface $D = (\ZZ/d\ZZ)^{n+1}$ and $P = S_{n+2}$. By \cref{prop:H1(P)_involutions} $H^1(\Gamma , S_{n+2})$ corresponds to involutions of $S_{n+2}$ modulo conjugation, that is, permutations of type $(1^l , 2^m)$ with $l+2m =n+2$. Furthermore, since $m \in \{0,1,\ldots, \left\lfloor \frac{n+2}{2} \right\rfloor\}$ we see that $H^1(\Gamma, S_{n+2})$ has $\left\lfloor \frac{n+2}{2} \right\rfloor + 1$ elements. Let $p_l \in S_{n+2}$ be an involution with $l$ fixed points of type $(1^l,2^m)$, actually, we can assume that $p_l = (1,2)(3,4)\cdots (2m-1,2m)$. Then the action of $\Gamma$ on $D=(\ZZ/d\ZZ)^{n+1}$ twisted by $p_l$ is
    $$ ^{\sigma_{p_l}} \mu := p_l \ ^{\sigma} \mu\  p_l^{-1} = p_l \ \mu^{-1} \ p_l^{-1} = \diag(\mu_{p_{l}(0)}^{-1} , \ldots,\mu_{p_{l}(n+1)}^{-1} )$$
    The condition of cocycle with the twisted action is $\mu \ ^{\sigma_{p_l}}\mu = 1$ in $D$, that is, $\mu_j \ \mu_{p_l(j)}^{-1} = \lambda$ for all $j$ and some scalar $\lambda \in \mu_d(\CC)$ (this is where the quotient by the diagonal $\Delta$ comes in). If we apply $p_l$ again, we obtain $\mu_{p_l(j)} \ \mu_j^{-1} = \lambda$, and multiplying both equations we conclude $\lambda^2 = 1$. Then, $\lambda \in \{-1,1\}$
    \begin{enumerate}
        \item If $d$ is odd, then $\lambda^2=1$ together with $\lambda^d=1$ implies $\lambda=1$. Since cocycles are of the form $(a_1 , a_1 , \ldots , a_m,a_m , b_1, \ldots , b_l)$ we see that two cocycles are equivalent if and only if there exists $c\in D$ such that $\mu_j' = c_j^{-1} \mu_j \ c_{p_l(j)}^{-1} $ and two cases appear:
        \begin{enumerate}
            \item If $(2r-1,2r)$ is a transposition of $p_l$, then $\mu_{2r-1}' = c_{2r-1}^{-1} a_r c_{2r}^{-1}$. Taking $c_{2r-1} = 1$ and $c_{2r}=a_r$ we obtain $\mu_{2r-1}'=1$.
            \item If $j$ is a fixed point of $p_l$, then $\mu_j' = c_j^{-2} b_j$, since $d$ is odd we see that the function $x\mapsto x^2$ is bijective in $\ZZ/d\ZZ$ and there exists $c_j$ such that $c_j^2=b_j$, then $\mu_j'=1$.
        \end{enumerate}
        Thus, all cocycles are equivalent to the trivial one, this means $H^1(\Gamma, \ _{p_l}D) = 1$ for all $l$. By \cref{thm:A} we conclude that 
        $$|H^1(\Gamma, \Aut(F^{n}_d))| = |H^1(\Gamma, S_{n+2})| = \left\lfloor \dfrac{n+2}{2} \right\rfloor + 1 = \left\lfloor \dfrac{n}{2} \right\rfloor + 2$$
        \item If $d$ is even, we have two cases for the value of $\lambda$:
        \begin{enumerate}
            \item If $\lambda = 1$, then the condition of cocycle is $\mu_j = \mu_{p_l(j)}$ similar as above and cocycles are of the form $(a_1, a_1, \ldots, a_m, a_m, b_1, \ldots, b_l)$. Now, two cocycles are equivalent if there exists $c\in D$ such that $\mu_j' = c_j^{-1}\mu_j c_{p_l(j)}^{-1}$ and we get two cases:
            \begin{enumerate}
                \item If $(2r-1,2r)$ is a transposition of $p_l$, then $\mu_{2r-1}' = c_{2r-1}^{-1} a_r c_{2r}^{-1}$. Taking $c_{2r-1} = 1$ and $c_{2r}=a_r$ we obtain $\mu_{2r-1}'=1$.
                \item If $j$ is a fixed point of $p_l$, then $\mu_j' = c_j^{-2} b_j$, since $d$ is even, the function $x\mapsto x^2$ is not surjective in $\ZZ/d\ZZ$, but for each fixed point $j$ of $p_l$ the class of $b_j$ in $\mu_d(\CC)/\mu_d(\CC)^2 \cong \ZZ/2\ZZ = \{-1,1\}$ is well defined; then the cocycles modulo coordinatewise coboundaries are identified with $(\ZZ/2\ZZ)^{l}$, identifying each cocycle with $(b_1 , \ldots, b_l) \in \{-1,1\}^{l}$.
            \end{enumerate}
            With the idea of applying \cref{thm:A}, we must compute the  orbits of $C_P(p_l)$ in the set $H^1(\Gamma, \ _{p_l}D)$. Let $c \in C_P(p_l)$; a direct computation shows that $c$ permutes fixed points into fixed points and transpositions into transpositions, so $c$ acts on $(b_1, \ldots, b_l)$ by a permutation of the $b_i$'s. Thus, two cocycles $(b_1 , \ldots , b_l)$ and $(b_1 ' , \ldots, b_l')$ lie in the same orbit if and only if they have the same number of entries equal to $1$. Using the fact that if $(b_1 , \ldots, b_l)$ has $r$ entries equal to $1$ then it is equivalent to $(-b_1 , \ldots , -b_l)$ with $l-r$ entries equal to $1$ (multiply the representative by the scalar matrix $-\Id\in\Delta$, which does not change the class in $D$, and renormalize the paired entries), we conclude that $r$ and $l-r$ are identified, which gives $\left\lfloor \dfrac{l}{2} \right\rfloor + 1$ orbits for each $p_l$.
            \item Now, suppose that $\lambda = -1$ and remark that this case occurs just if $p_l$ does not have fixed points then $l=0$ since if $j$ is a fixed point of $p_l$ the cocycle condition forces $\lambda=1$, a contradiction. Then, these cocycles are of the form $(a_1 , -a_1 , \ldots, a_m , -a_m)$ and two cocycles $\mu$ and $\mu'$ are equivalent if and only if there exists $c\in D$ such that $\mu_j' = c_j^{-1} \ \mu_j \ c_{p_l(j)}^{-1}$. But note that
            $$ \mu_j' (\mu_{p_0(j)}')^{-1} = (c_j^{-1} \mu_j c_{p_0(j)}^{-1})(c_{p_0(j)} \mu_{p_0(j)}^{-1} c_j) = \mu_j \mu_{p_0(j)}^{-1} $$
            Then, cocycles with $\lambda = 1$ cannot be equivalent to cocycles with $\lambda =-1$. Moreover, normalizing the pairs as above, every cocycle with $\lambda=-1$ is equivalent to $(1,-1,\ldots,1,-1)$, so these cocycles form exactly one class. Thus, we obtain an additional class in the fiber of $[p_0] \in H^1(\Gamma, S_{n+2})$.
        \end{enumerate}
        Now \cref{thm:A} says
        $$|H^1 (\Gamma, \Aut(F^n_d))| = \sum_{[p_l] \in H^1(\Gamma, S_{n+2})} \left(\left\lfloor\frac{l}{2}\right\rfloor + 1\right) $$
        where $[p_0]$ gives one additional class when $n$ is even, coming from $\lambda =-1$. In this context, the count is more explicit depending on the parity of $n$:
        \begin{itemize}
            \item If $n$ is odd and $l\in \{1,3, \ldots, n+2 \}$, i.e., $l=2i+1$ with $i \in \left\{0,1,\ldots, \frac{n+1}{2}\right\}$. Then, 
            $$ \sum_{i=0}^{\frac{n+1}{2}}\left(\left\lfloor\frac{2i+1}{2}\right\rfloor + 1\right) =\sum_{i=0}^{\frac{n+1}{2}} (i+1) = \frac{(n+3)(n+5)}{8}$$
            \item If $n$ is even and $l\in \{ 0,2,\ldots, n+2 \}$ we have $l=2i$ for $i \in \{0,1,\ldots, \frac{n}{2} +1\}$. Since $i=0$ has two classes (for each $\lambda$), then
            $$1  + \sum_{i=0}^{\frac{n}{2} +1} \left(\left\lfloor\frac{2i}{2}\right\rfloor + 1\right) =  1 + \sum_{i=0}^{\frac{n}{2}+1}(i+1) = 1 + \frac{(n+4)(n+6)}{8}$$
        \end{itemize}
    \end{enumerate}
    This proves the statement.
\end{proof}
In the next subsection we show that the same framework applies, under the hypotheses of \cref{rem:standing}, to finite fields $k$.

\subsection{\texorpdfstring{$\FF_q$-forms}{Fq-forms}} Let $p \in \ZZ$ be an odd prime number and $q$ a power of $p$. Taking $k = \FF_q$ and $K = \overline{\FF}_q$ we obtain that $\Gal(K/k) \cong \widehat{\ZZ}$ which, unlike the case $k=\RR$, is not finite; however it is procyclic and, more importantly, it is a profinite group, so the statements of the previous section apply. Indeed, $\widehat{\ZZ}$ is topologically generated by the Frobenius homomorphism given by
\begin{align*}
    \phi : \overline{\FF}_q &\to\overline{\FF}_q\\
    x &\mapsto x^q
\end{align*}

As we proved in \cref{sec:diff}, the differential method extends to positive characteristic and, under the hypotheses of \cref{rem:standing}, the automorphism groups of our family of hypersurfaces over $K=\overline{\FF}_q$ are the same as over $\CC$.

In a similar way we can define $\phi=1 \times \sigma$ as an $\FF_q$-structure and state \cref{prop:bijection_forms_cohomology} in this context as follows
\begin{theorem}[\cite{Se}]
    Let $X$ be a quasiprojective algebraic variety defined over $\overline{\FF}_q$ with an $\FF_q$-structure $\phi$. There is a natural bijection between the set of $\FF_q$-forms of $X$ and $H^1(\Gamma ,\Aut_{\overline{\FF}_q}(X))$, where the topological generator $\sigma$ of $\Gamma = \Gal(\overline{\FF}_q / \FF_q)$ acts on $\Aut_{\overline{\FF}_q}(X)$ by $^{\sigma} \psi = \phi \ \psi \ \phi^{-1}$.
\end{theorem}

Indeed, if $X$ is a hypersurface such that $\Aut(X) \cong D \rtimes P$, the actions of $\Gamma$ on $D$ and $P$ are given in the following lemma.
\begin{lemma}{\label{lem:action_of_Gamma_over_Fq}}
    Let $\phi : \PP^{n+1}_{\overline{\FF}_q} \to \PP^{n+1}_{\overline{\FF}_q}$, given by $\phi(x) = x^q$, be the Frobenius morphism induced on $\PP^{n+1}_{\overline{\FF}_q}$ by $\sigma$. Let $\mu = \diag(\mu_0 , \mu_1 , \ldots, \mu_{n+1}) \in D$ and $c \in P$. Then, 
    \begin{enumerate}
        \item $^{\sigma} \mu  = \phi \ \mu \ \phi^{-1} = \mu^{q}$;
        \item $^{\sigma} c = \phi \ 
        c\ \phi^{-1} = c$;
        \item $c \ \mu \ c^{-1} = \diag(\mu_{c^{-1}(0)} , \ldots, \mu_{c^{-1}(n+1)})$.
    \end{enumerate}
\end{lemma}
\begin{proof}
    Since $\overline{\FF}_q$ is a perfect field, $\phi$ is bijective and $\phi^{-1} ([x_0 : x_1 : \cdots : x_{n+1}]) = [x_0 ^{\frac{1}{q}} : x_1^{\frac{1}{q}} : \cdots : x_{n+1}^{\frac{1}{q}}]$. Now, 
    \begin{enumerate}
        \item Note that 
        \begin{align*}
            \phi \mu ( \phi^{-1} \cdot [x_0 : x_1 : \cdots : x_{n+1}]) &= \phi (\mu \cdot [x_0 ^{\frac{1}{q}} : x_1^{\frac{1}{q}} : \cdots : x_{n+1}^{\frac{1}{q}}] )\\
            &= \phi \cdot [\mu_0 \ x_0 ^{\frac{1}{q}} : \mu_1 \ x_1^{\frac{1}{q}} : \cdots :\mu_{n+1} \ x_{n+1}^{\frac{1}{q}}]\\
            &= [\mu_0^q \ x_0 : \mu_1^{q} \ x_1 : \cdots : \mu_{n+1}^{q} \ x_{n+1} ]
        \end{align*}
        But this is equivalent to act on homogeneous coordinates by $\mu^q$.
        \item This is a consequence of \cref{prop:equivariant_section} taking $k = \FF_q$.
        \item The same argument given in \cref{prop:action_of_Gamma} holds here.
    \end{enumerate}
\end{proof}

Now, assume that $X$ has automorphism group isomorphic to $D \rtimes P$, with $D$ an abelian group of diagonal automorphisms and $P \leq S_{n+2}$.
\begin{proposition}{\label{prop:H1(D)_descomposition_over_Fq}}
    $H^1(\Gamma, D)$ is identified with $ D/(1-\phi)D$.
\end{proposition}
\begin{proof}
    In \cite[Chapter I, \S5.1, Exercise 2]{Se} it is mentioned that if we take a cocycle $(a_s)$ of $\Gamma$ on some $\Gamma$-group $A$, letting $a= a_{\sigma}$ with $\sigma$ the topological generator of $\Gamma \cong \widehat{\ZZ}$ one can prove that there exists $n \geq 1$ such that $\sigma^n(a)=a$ and that $a \cdot \sigma(a) \cdots \sigma^{n-1}(a)$ is of finite order. Conversely, every $a \in A$ for which there exists such an $n$ corresponds to one and only one cocycle. Moreover, if $a$ and $a'$ are two such elements, the corresponding cocycles are equivalent if and only if there exists $b \in A$ such that $a' = b^{-1} \cdot a \cdot \sigma(b)$.
    Since $D$ is a finite abelian $\Gamma$-group, every $a \in D$ satisfies both conditions above for some $n$, then $Z^1 (\Gamma, D)$ is identified with $D$, and by \cref{lem:action_of_Gamma_over_Fq}, the equivalence relation $a'=b^{-1}\cdot a\cdot\phi(b)$ becomes $a' = b^{-1} a b^{q}$. Again, since $D$ is abelian, this says $a' = a b^{q-1}$. Hence, $a$ is equivalent to $a'$ if and only if $a'a^{-1} \in \{b^{1-q} : b \in D\}$, this amounts to saying that $a'a^{-1} \in (1- \phi)D $ and we conclude the result.
\end{proof}

\begin{proposition}{\label{prop:H1(P)_over_Fq}}
    $H^1(\Gamma ,P)$ is identified with the set of classes of $P$ modulo conjugation.
\end{proposition}
\begin{proof}
    By \cref{prop:equivariant_section}, we know that the action of $\Gamma$ on $P$ is always trivial (in any field), then $Z^1(\Gamma  , P) = \operatorname{Hom}_{\operatorname{cont}}(\Gamma  , P) \cong P$. Then, two cocycles $a$ and $a'$ are equivalent if and only if there exists $b \in P$ such that $a' = b^{-1} \ a \ b$ and we are done.
\end{proof}
In particular, we can state \cref{thm:A} in this context as:
\begin{theorem}{\label{teo:fq-general}}
    Let $X \subset \PP^{n+1}_{\overline{\FF}_q}$ be one of the
    hypersurfaces $F^n_d$, $T^n_d$ or $K^n_d$, of dimension $n \geq 2$ and
    degree $d \geq 3$ with $(n,d) \neq (2,4)$, and assume that the
    corresponding hypothesis of \cref{rem:standing} holds. Then the number of $\FF_q$-forms of $X$ is given by
    $$ |H^1(\Gamma , \Aut(X))| = \sum_{[c] \in H^1(\Gamma, P)} |H^1(\Gamma,\ _{c}D)/C_{P}(c)|  $$
    where $H^1(\Gamma, \ _{c}D) / C_P (c) $ denotes the set of orbits of the action of $C_P(c)$ on $H^1(\Gamma, \ _{c}D)$.
\end{theorem}
\begin{proof}
    By \cref{rem:standing} we have $\Aut(X)\cong D\rtimes P$ as in
    \cref{cor:DrtimesP}, with the same groups $D$ and $P$ as over $\CC$, so
    this is \cref{thm:A} applied to $k=\FF_q$.
\end{proof}
Remark that the hypotheses of \cref{rem:standing} allow us to assume that $\Aut(X)$ has the same description $D \rtimes P$ as over $\CC$ for the family of hypersurfaces used below.

\subsubsection{Count of \texorpdfstring{$\FF_q$}{Fq}-forms of classical hypersurfaces} Let $T^{n}_d \subset \PP^{n+1}_{\overline{\FF}_q}$ be the Delsarte hypersurface of dimension $n$ and degree $d$. Let $p$ be a prime such that $p \nmid d(d-1)(d-2)$, then, by \cref{prop:delsarte} we see that if $n\geq 2$ and $d\geq 4$ with $(n,d) \neq (2,4)$, $\Aut(T^n_d) \cong \ZZ/(d-1)^{n+1}\ZZ$.

\begin{corollary}\label{cor:Fq_forms_Delsarte}
    Let $T^n_d \subset \PP^{n+1}_{\overline{\FF}_q}$ be the Delsarte hypersurface of dimension $n\geq 2$ and degree $d \geq 4$ with $(n,d)\neq (2,4)$, and $p\nmid d(d-1)(d-2)$. Then the number of $\FF_q$-forms is
    $$ |H^1(\Gamma , \Aut(T^n_d))| = \gcd(q-1, (d-1)^{n+1}) $$
\end{corollary}

\begin{proof}
    As $\Aut(T^n_d) \cong \ZZ/(d-1)^{n+1}\ZZ$, defining $N = (d-1)^{n+1}$ we see by \cref{prop:H1(D)_descomposition_over_Fq} that $H^1(\Gamma , \Aut(T^n_d)) = (\ZZ/N\ZZ) / (1 - \phi)(\ZZ/N\ZZ)$. Using the fact that $(1-\phi)(\ZZ/N\ZZ)$ is a subgroup of $\ZZ/N\ZZ$ of index $\gcd(q-1,N)$, we conclude $H^1 (\Gamma , \Aut(T^n_d)) \cong  \ZZ/\gcd(q-1,N)\ZZ$.
\end{proof}

Now, consider the Klein hypersurface $K^{n}_d \subset \PP^{n+1}_{\overline{\FF}_q}$ of dimension $n$ and degree $d$, and $p$ a prime such that $p \nmid d(d-1)(d-2)m$; in this way, by \cref{prop:klein}, $\Aut(K^n_d) \cong \ZZ/m\ZZ \rtimes \ZZ/(n+2)\ZZ$.
\begin{corollary}\label{cor:Fq_forms_Klein}
    Let $K^{n}_{d} \subset \PP^{n+1}_{\overline{\FF}_q}$ be the Klein hypersurface of dimension $n \geq 2$ and degree $d \geq 4$ with $(n,d)\neq (2,4)$, and $p \nmid d(d-1)(d-2)m$. Then, the number of $\FF_q$-forms is
    $$ |H^1(\Gamma , \Aut(K^{n}_{d}))| = \sum_{j=0}^{n+1} |H^1 (\Gamma  , \ _{\tau^{j}}D) / \langle \tau \rangle| $$
    where $|H^1(\Gamma , \ _{\tau^{j}} D)| = \gcd(m, q(1-d)^{-j}-1)$, and $\langle \tau \rangle$ act on $H^1(\Gamma, \ _{\tau^j}D)$ by $^{\tau^{k}}\mu = \mu^{(1-d)^k}$.
\end{corollary}

\begin{proof}
    Here, $D=\langle \mu\rangle\cong\ZZ/m\ZZ$, where $\mu=\diag(\mu_0,\ldots,\mu_{n+1})$ with $\mu_j=\mu_0^{(1-d)^j}$, and $P=\langle \tau\rangle\cong \ZZ/(n+2)\ZZ$, with $\tau=(0\ 1\ \cdots\ n\ n+1)$. Since $P$ is an abelian group, the conjugation class of $c \in P$ is $\{c\}$ and by \cref{prop:H1(P)_over_Fq} we see that $ H^1(\Gamma,P) = \{[\tau^j] : j=0,1,\ldots,n+1\}$.
    In this way, we obtain $n+2$ classes, and $C_P(\tau^j) = P$ for all $j \in \{0,1,\ldots, n+1\}$.
    By \cref{lem:action_of_Gamma_over_Fq}, we have $c \ \mu \ c^{-1} = \diag\left(\mu_{c^{-1}(0)} , \ldots, \mu_{c^{-1}(n+1)}\right)$, then in each entry
    $$ \tau\mu_j\tau^{-1} = \mu_{\tau^{-1}(j)} = \mu_{j-1}  $$
    Since $\mu_j = \mu_0^{(1-d)^j}$, it follows that $\mu_{j-1} = \mu_0^{(1-d)^{j-1}} = \mu_j^{(1-d)^{-1}}$, i.e., $\tau\mu\tau^{-1} = \mu^{(1-d)^{-1}}$. In particular, $\tau^{j}\mu\tau^{-j} = \mu^{(1-d)^{-j}}$.
    If we twist by $\tau^j$, using \cref{lem:action_of_Gamma_over_Fq} we obtain $ ^{\sigma_{\tau^j}}(\mu) = \tau^j \cdot \ ^{\sigma}\mu \cdot \tau^{-j} = (\tau^j\mu\tau^{-j})^q = \mu^{q(1-d)^{-j}}$. This says that $\Gamma$ acts on $_{\tau^j}D$ multiplying by $\lambda_j := q(1-d)^{-j}$; identifying, as before, cocycles for this twisted action with the set $D$, two cocycles $a$ and $a'$ are equivalent if $a' = a\cdot(1-\lambda_j)(b)$ for some $b\in D$, then
    $$ H^1(\Gamma,\ _{\tau^j}D) \cong (\ZZ/m\ZZ)/(1-\lambda_j)(\ZZ/m\ZZ). $$
    Since multiplying by $k$ on $\ZZ/m\ZZ$ has image with index $\gcd(k,m)$, we conclude 
    $$ |H^1(\Gamma,\ _{\tau^j}D)| = \gcd\left(m,\ q(1-d)^{-j}-1\right) $$
    And finally, the action of $C_P(\tau^j)=\langle\tau\rangle$ on $H^1(\Gamma,\ _{\tau^j}D)$ is $\tau^{k}\cdot[\mu^{r}] = [\tau^{-k}\mu^{r}\tau^{k}] = [\mu^{r(1-d)^{k}}]$, this shows that $\langle\tau\rangle$ acts on $H^1(\Gamma,\ _{\tau^j}D)\cong \ZZ/\gcd(m,q(1-d)^{-j}-1)\ZZ$ multiplying by $(1-d) \mod \gcd(m,q(1-d)^{-j}-1)$. Adding over the $n+2$ classes $[\tau^j]\in H^1(\Gamma,P)$, the corollary follows.
\end{proof}

Finally, we apply this formula to the context of the Fermat hypersurface $F^n_d$ of dimension $n$ and degree $d$ and taking a prime number $p$ such that $p\nmid d(d-1)$, then, by \cref{prop:fermat}, $\Aut(F^n_d) \cong (\ZZ/d\ZZ)^{n+1} \rtimes S_{n+2}$.

\begin{corollary}\label{cor:Fq_forms_Fermat}
    Let $F^{n}_d \subset \PP^{n+1}_{\overline{\FF}_q}$ be the Fermat hypersurface of dimension $n\geq 2$ and degree $d\geq 3$, with $(n,d)\neq(2,4)$ and $p\nmid d(d-1)$. Then, the number of $\FF_q$-forms is
    $$ |H^1(\Gamma , \Aut(F^n _d ))| = \sum_{[c_{\lambda}] \in H^1(\Gamma , S_{n+2})} |H^1(\Gamma , \ _{c_{\lambda}}D) / C_{S_{n+2}}(c_{\lambda})| $$
    where, for $c_{\lambda} \in S_{n+2}$ of type $\lambda =(1^{l} , \lambda_1 , \ldots, \lambda_r)$ with $\lambda_j \geq 2$ and $l+ \sum_{j=1}^{r} \lambda_j = n+2$,
    $$ |H^1(\Gamma , \ _{c_{\lambda}}D)| = \frac{t_\lambda}{d}\,\gcd(d,q-1)^{l} \prod_{j=1}^{r}\gcd(d,q^{\lambda_j}-1),
    \qquad t_\lambda:=\gcd\Big(d,\ \gcd_{s\in L_\lambda}\big(\sigma_s\, d/g_s\big)\Big), $$
    with $g_s := \gcd(d,q^s-1)$, $\sigma_s := 1+q+\cdots+q^{s-1}$, and $L_\lambda$ the set of cycle lengths of $c_\lambda$ (including $s=1$ when $l\geq1$).
\end{corollary}

\begin{proof}
    We know that $H^1(\Gamma , S_{n+2})$ corresponds to conjugacy classes of $S_{n+2}$. Moreover, if we take $[c_{\lambda}] \in H^1(\Gamma , S_{n+2})$ of type $\lambda = (1^{l} , \lambda_1 , \ldots , \lambda_{r})$ with $\lambda_j \geq 2$ and $l+\sum_{j=1}^{r} \lambda_j = n+2 $, then
    $$ ^{\sigma_{c_{\lambda}}}\mu = \diag(\mu_{c_{\lambda}^{-1}(0)} ^{q} , \ldots ,  \mu_{c_{\lambda}^{-1}(n+1)} ^{q}) $$
    Recall from  \cref{rmk:tilde_D} that $D=\widetilde D/\Delta$,
    where $\widetilde D=\mu_d(K)^{n+2}\cong(\ZZ/d\ZZ)^{n+2}$ is the group of
    diagonal matrices with entries in $\mu_d(K)$ and $\Delta\cong\ZZ/d\ZZ$ is
    the subgroup of scalar matrices. By
    \cref{prop:H1(D)_descomposition_over_Fq} we have
    $H^1(\Gamma,\ _{c_\lambda}D)=D/(1-\phi_{c_\lambda})D$, where
    $\phi_{c_\lambda}([\mu])_j=\mu_{c_\lambda^{-1}(j)}^q$, and since
    $1-\phi_{c_\lambda}$ is an endomorphism of the finite abelian group $D$
    we get $|H^1(\Gamma,\ _{c_\lambda}D)|=|\ker(1-\phi_{c_\lambda})|$, which
    we now compute. Writing $\widetilde D=(\ZZ/d\ZZ)^{n+2}$ additively, a
    class $[\mu]\in D$ is fixed by $\phi_{c_\lambda}$ if and only if there
    exists a scalar $t\in\ZZ/d\ZZ$ such that
    $$ q\,\mu_{c_\lambda^{-1}(j)}=\mu_j+t
       \qquad\text{for all } j\in\{0,\ldots,n+1\};$$
    this is where the quotient by $\Delta$ comes in, exactly as the scalar
    $\lambda$ did in the real case. The scalar $t$ is determined by $\mu$,
    and each class $[\mu]$ has exactly $d$ representatives in
    $\widetilde D$, so
    $$ |\ker(1-\phi_{c_\lambda})|=\frac1d\sum_{t\in\ZZ/d\ZZ}
       \#\{\mu\in(\ZZ/d\ZZ)^{n+2}\mid
       q\,\mu_{c_\lambda^{-1}(j)}=\mu_j+t\ \ \forall j\}. $$
    For a fixed $t$ the equations decouple along the cycles of $c_\lambda$.
    On a cycle $(i,c_\lambda(i),\ldots,c_\lambda^{s-1}(i))$ of length $s$
    they give $\mu_{c_\lambda^{k}(i)}=q^k\mu_i-\sigma_k\,t$ for
    $1\leq k\leq s-1$, and closing the cycle,
    $$ (q^s-1)\,\mu_i=\sigma_s\,t,
       \qquad \sigma_s=1+q+\cdots+q^{s-1}. $$
    This congruence has solutions $\mu_i\in\ZZ/d\ZZ$ if and only if
    $g_s=\gcd(d,q^s-1)$ divides $\sigma_st$, that is, if and only if
    $(\sigma_s\, d/g_s)\,t\equiv0 \bmod d$, in which case there are exactly
    $g_s$ of them. Hence the admissible scalars form the subgroup
    $$T_\lambda=\{t\in\ZZ/d\ZZ \mid (\sigma_s\, d/g_s)\,t\equiv0 \bmod d
    \ \text{ for all } s\in L_\lambda\},$$
    whose order is $t_\lambda$, and for each $t\in T_\lambda$ there are
    exactly $\gcd(d,q-1)^l\prod_{j=1}^r\gcd(d,q^{\lambda_j}-1)$ solutions
    $\mu$ (one factor $g_s$ for each cycle, the $l$ fixed points being the
    cycles with $s=1$). Therefore
    $$ |H^1(\Gamma,\ _{c_\lambda}D)|=\frac{t_\lambda}{d}\,
       \gcd(d,q-1)^{l}\prod_{j=1}^{r}\gcd(d,q^{\lambda_j}-1), $$
    and summing over $[c_\lambda]\in H^1(\Gamma,S_{n+2})$ via
    \cref{teo:fq-general} the corollary follows.
\end{proof}

\begin{remark}
  The factor $t_\lambda/d$ accounts for the quotient by the scalar
  matrices $\Delta$: forgetting it, i.e., computing in
  $\widetilde D=\mu_d(K)^{n+2}$ instead of $D$, would give
  $\gcd(d,q-1)^{l}\prod_{j}\gcd(d,q^{\lambda_j}-1)$, which is incorrect in
  general.
\end{remark}

\bibliographystyle{alpha}
\bibliography{nfolds}
\end{document}